\IfFileExists{./prepreamble-ip.sty}{\RequirePackage[packages,theorems]{prepreamble-ip}}{}

\documentclass[]{amsart}

\RequirePackage{biblatex}
\RequirePackage{ltxcmds}
\IfFileExists{./preamble-ip.sty}{\RequirePackage[packages,theorems]{preamble-ip}}{}

\usepackage{manuscript}
\usepackage{KMS-local}
\usepackage{KMS-standard-packages}
\usepackage{scoop-semantic}

\IfFileExists{./postpreamble-ip.sty}{\RequirePackage[packages,theorems]{postpreamble-ip}}{}

\makeatletter
\@ifpackageloaded{hyperref}{%
	\hypersetup{
		pdftitle = {Hybrid CG-Tikhonov is a filtration of the CG Lanczos vectors},
		pdfauthor = {Daniel Gerth, Kirk M. Soodhalter},
		pdfkeywords = {Krylov subspace methods, Conjugate Gradients, Ill-posed problems, Regularization methods},
		pdfcreator = {Created using the Scoop Template Engine version 1.6.0.}
	}
}{
	\pdfinfo{
		/Title (Hybrid CG-Tikhonov is a filtration of the CG Lanczos vectors)
		/Author (Daniel Gerth, Kirk M. Soodhalter)
		/Subject ()
		/Keywords (Krylov subspace methods, Conjugate Gradients, Ill-posed problems, Regularization methods)
		/Creator (Created using the Scoop Template Engine version 1.6.0.)
	}
}
\makeatother

\title[Relationship of CG and CG-Tikhonov]{Hybrid CG-Tikhonov is a filtration of the CG Lanczos vectors}

\author[D. Gerth]{Daniel Gerth}
\address[D. Gerth]{Technische Universität Chemnitz, (formerly; now in private sector)}
\email{\detokenize{dgerth123@proton.me (private)}}

\author[K. M. Soodhalter]{Kirk M. Soodhalter}
\address[K. M. Soodhalter]{Trinity College Dublin, College Green, Dublin 2 Ireland}
\email{\detokenize{ksoodha@maths.tcd.ie}}
\urladdr{https://maths.soodhalter.com}

\thanks{The first author was kindly supported through DFG Grant number GE 3171/1-1.}

\date{\today}

\dedicatory{}

\begin{document}

\begin{abstract}
We consider iterative methods for solving linear ill-posed problems with compact operator and right-hand side only available via
noise-polluted measurements. Conjugate gradients (\cg) applied to the normal equations with an appropriate stopping rule and \cg applied
to the system solving for a Tikhonov-regularized solution (\cgt) $(A^\ast A + c I_{\mathcal{X}}) x^{(\delta,c)} = A^\ast y^\delta$
are closely related regularization methods that build iterates from the same Krylov subspaces.

In this work, we show that the \cgt iterate can be expressed 
\linebreak
as
\begin{math}
x^{(\delta,c)}_m 
=
\sum_{i=1}^{m} \gamma^{(m)}_i(c) z_i^{(m)}v_i,
\end{math}
where $\braces{\gamma_i^{(m)}(c)}_{i=1}^m$ are functions of the Tikhonov parameter $c$ and $x^{(\delta)}_m = \sum_{i=1}^{m}
z_i^{(m)}v_i$ is the $m$-th \cg iterate. We call these functions \emph{Lanczos filters}, and they can be shown to have decay
properties as $c\rightarrow\infty$ with the speed of decay increasing with $i$. This has the effect of filtering out the
contribution of the later terms of the \cg iterate. The filters can be constructed using quantities defined via recursions at each 
iteration.

We demonstrate with numerical experiments that good parameter choices correspond to appropriate damping of the Lanczos vectors.
The filtration approach also provides a platform for further development of parameter choice rules, and similar representations
may hold for other hybrid iterative schemes.  

\end{abstract}


\makeatletter
\ltx@ifpackageloaded{hyperref}{%
\subjclass[2010]{\href{https://mathscinet.ams.org/msc/msc2020.html?t=65F22}{65F22}, \href{https://mathscinet.ams.org/msc/msc2020.html?t=65F10}{65F10}, \href{https://mathscinet.ams.org/msc/msc2020.html?t=65F20}{65F20}, \href{https://mathscinet.ams.org/msc/msc2020.html?t=45Q05}{45Q05}}
}{%
}
\makeatother

\maketitle

\section{Introduction}
\label{section:introduction}
In this paper, we consider the linear operator approximation problem
\begin{align}
	Ax
	\approx 
	y,
	\label{eqn.AxAppy}
\end{align}
where $A$ is a compact operator between two real, separable Hilbert spaces $\cX$
and $\cY$. 
We solve the normal equations 
\begin{align}
    	A^{\ast}Ax^\dagger
	=
	A^{\ast}y
	\label{eqn.ne}
\end{align}  
to obtain the minimum norm pseudoinverse solution.
In the case in which we only possess 
	$
	y^\delta
	=
	y + \delta n
	$ 
	due to measurement error, 
where $n\in\cY$ is a unit norm error element and $\delta > 0$ is the measurement error level
(sometimes called the noise level), it is not recommended to apply the pseudoinverse.
As the problem is ill-posed, computing an approximation 
	$
	x^\delta
	=
	\left(  A^{\ast}A \right)^{-1}A^\ast y^\delta
	$ 
	to $x^\dagger$ 
in the presence of measurement error in the right-hand side can lead to
unbounded errors in the approximate solution. Rather, one seeks to compute a
stable approximation using a regularization method. We study two well-known
methods of obtaining a regularized solution: the method of conjugate gradients
(\cg) applied to the normal equations (\cgne) equipped with an appropriate
stopping criterion and \cg-based Tikhonov regularization (\cgt) equipped with
an appropriate parameter selection and criterion for determining an appropriate
number of iterations. We note to the reader that in the setting of this paper,
we consider \cg and \cgt for \emph{fixed parameter choices}, focusing on
understanding the relationships of these methods.  However, we do relate our
work to parameter choice rules when discussing future work in
\Cref{section.conclusions}.  

Both methods can be derived using the self-adjoint Lanczos iteration and their iterates
can thus be related using some linear algebra applied to the matrices which arise from the Lanczos iteration. The goal of this
work is to relate the behavior of these two methods in the context of their properties as regularizers.  For much of this paper,
we consider these methods applied in the infinite dimensional setting of \eqref{eqn.AxAppy}.  Thus, to further our aims, we extend
results of the recent paper \cite{AlqahtaniRamlauReichel:2022:1} to our setting.

Our analysis highlights an alternative approach to understanding the effect of applying iterative regularization techniques, via
\emph{the filtration of the basis vectors generated by the method rather than of the singular vectors of $A$}.  Although we focus on
relating \cgne and \cgt, we propose that this view will have utility for other iterative and hybrid regularization techniques. 

\subsection{Organization and notation of this paper}\label{subsection:organization} This paper is organized as follows.  In
\Cref{section:regularization-bg}, we present an overview of relevant regularization theory, classical filter functions, and
Tikhonov regularization.   In \Cref{section:Lanczos-CG-CGT}, we give a brief derivation of conjugate gradients, focusing on its
application to the normal equations and for Tikhonov regularization.  In \Cref{section:tri-diag-inverse}, we discuss the structure
of inverses of tridiagonal matrices in order to relate \cg applied to the normal equations and to Tikhonov functional minimization.
In \Cref{section:GKB-Lanczos-inf-dim}, we discuss the practical implementation of \cg using Golub-Kahan bidiagonalization, and we
discuss some relevant properties of this iteration in infinite dimensions.  In \Cref{section:filtration-lanczos-vectors}, we
leverage the analysis we develop to present the concept of Lanczos vector filtration. In \Cref{section.numerical-examples}, we
present one additional numerical example using the \chebfun version \cite{AlqahtaniMachReichel:2022:1} of the \texttt{gravity} example from \cite{Hansen:2007:1}. In \Cref{section.conclusions}, we conclude by suggesting a number of
avenues of research that can be built on the foundations of this work.

When not otherwise indicated, capital letters are used to refer to operators or matrices. Lower-case letters are used to refer to
elements in a vector space. For clarity, we differentiate between abstract operator and vector quantities and those that are
represented in a basis.  Quantities (both infinite and finite dimensional) that are represented in a specific basis (such as a
matrix or a vector expressed as a list of components) are denoted in \textbf{boldface}.  We illustrate some concepts in the text
by computing some steps of \cgne and \cgt on a continuous version of the \texttt{shaw} problem from the Regularization Toolbox
\cite{Hansen:2007:1}.  The continuous version of this problem was implemented using \chebfun
\cite{DriscollHaleTrefethen:2014:1,TownsendTrefethen:2013:1} by the authors of \cite{AlqahtaniMachReichel:2022:1}.  They point out
that the \chebfun allows one to implement continuous regularization techniques, with approximation accuracy up to near
machine precision. They provide a suite of functions\footnote{\url{https://github.com/thomasmach/Ill-posed_Problems_with_Chebfun}} 
implementing these problems, and we build our experiments on these
functions.  For the illustrative experiments in the text, we call $\mathtt{[A,x,btrue] = shaw\_chebfun()}$.  This generates
\chebfun approximation of the Fredholm integral equation of the first kind
\begin{align*}
	\int_{-\pi/2}^{\pi/2}k(s,t)f(t)dt 
	= 
	g(s)
	\qquad
	\mbox{where}&
	\qquad
	k(s,t) 
	=
	\prn
	{
		\cos(s)+\cos(t)
	}
	\prn
	{
		\dfrac{
			\sin(u)
		}
		{
			u
		}
	}^2,
	\\
	\mbox{and}&
	\qquad
	u=\pi (\sin(s)+\sin(t))
\end{align*}
posed on the symmetric interval $\paren[auto]{[}{]}{-\frac{\pi}{2},\frac{\pi}{2}}$.  
In our demonstrations, we use a relative noise level of $10^{-4}$. 
In \Cref{fig:shaw-example}, we
plot the data generated for this example problem.
\begin{figure}
	\includegraphics[scale=0.3]{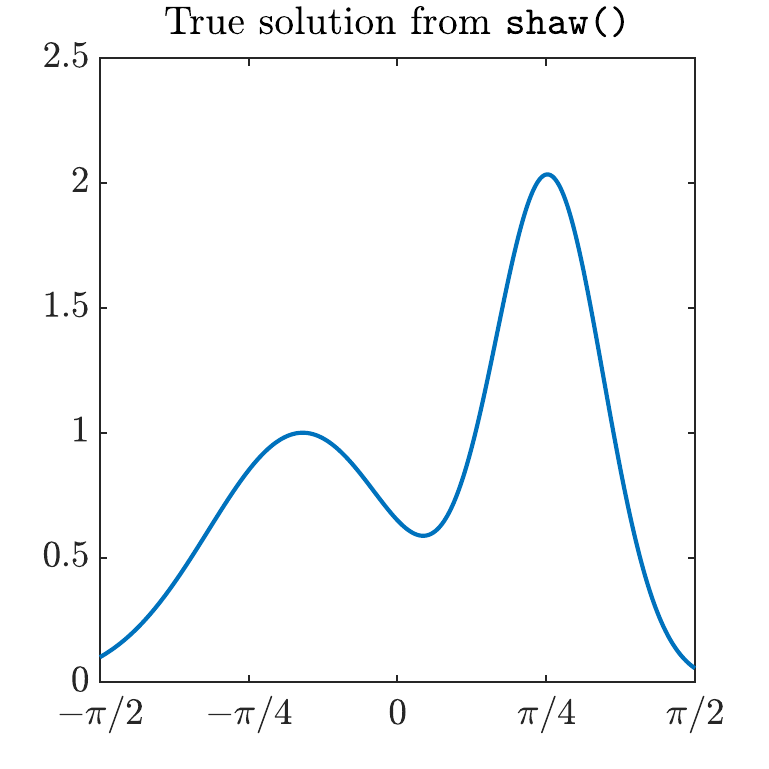}
	\includegraphics[scale=0.3]{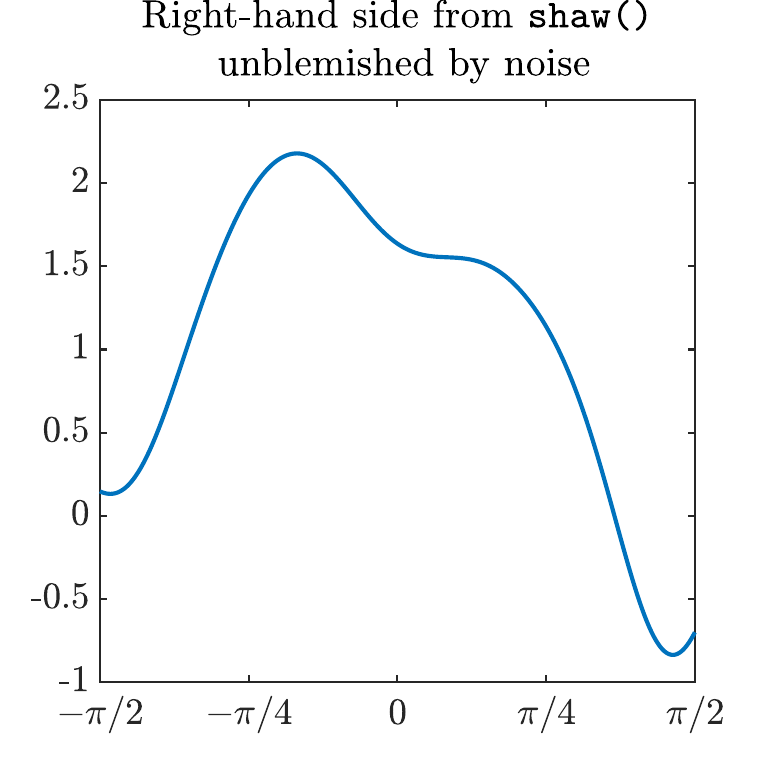}
	\includegraphics[scale=0.3]{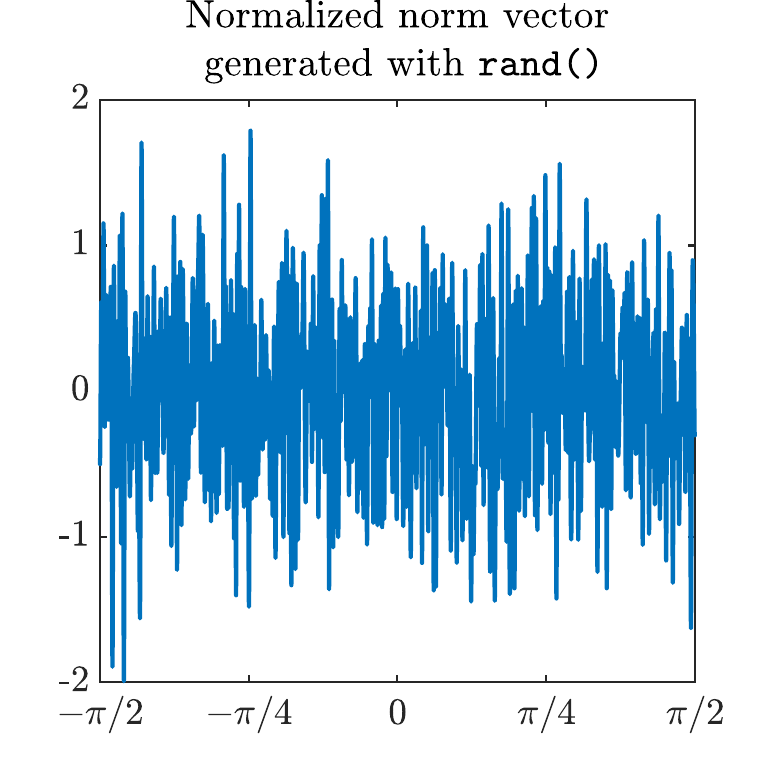}
	\caption{Solution, right-hand side, and noise vectors for $\mathtt{shaw\_chebfun()}$\label{fig:shaw-example}}
\end{figure}
Additional experiments are presented in \Cref{section.numerical-examples}.

\section{Regularization and filter functions}
\label{section:regularization-bg}
The theory of regularization has a rich history beyond the scope of
this paper.  We point the reader to much deeper texts on the topic, such as
\cite{EnglHankeNeubauer:1996:1}.  Conceptually, a regularization technique
seeks to approximate $x^\dagger$ using $y^\delta$ in such a way that some fidelity
to the right-hand side is maintained while the destabilizing effects of $\delta n$
are damped.

Since the operator $A$ is a compact linear map between two Hilbert spaces, 
we know that it has a singular system.  We assume that $A$ is not degenerate;
thus, we can write its singular system $\braces{ \sigma_i, z_i,w_i }_{i=0}^\infty$
where $\braces{ w_i}_{i=0}^\infty\subset \cX$ and 
$\braces{ z_i}_{i=0}^\infty\subset \cY$ are orthonormal
systems in $\cX$ and $\cY$, respectively, and we have $\lim_{i\rightarrow\infty}\sigma_i
	=
	0$.  
We can express the actions of $A$ and its adjoint, respectively, using the \svd with  
\begin{align}\label{eqn.svd-A-mapping}
	A:x\mapsto \sum_{i=0}^\infty \sigma_i \left( x, w_i \right)_\cX z_i\qquad\mbox{and}\qquad A^\ast:y\mapsto\sum_{i=0}^\infty \sigma_i \left( y, z_i \right)_\cY w_i.
\end{align}
It follows that we can also use this system to represent the action of the normal equations operator,
\begin{align}\label{eqn.svd-Ane-mapping}
	A^\ast A: x\mapsto\sum_{i=0}^\infty\sigma_i^2 \left( x, w_i \right)_\cX w_i.
\end{align}
The pseudoinverse solution can be expressed in terms of the singular system, namely
\begin{align}
	x^\dagger
	=
	\sum_{i=0}^\infty \dfrac{\left( y, z_i \right)_\cY}{\sigma_i} w_i.
	\label{eqn:SVD-pseudoinverse-soln}
\end{align}
As we generally assume that the measurement perturbation $n\in\cY$ has nontrivial components
in the entire system of singular vectors, inserting $y^\delta$ into the formula for $x^\delta$ will
result in $\norm[auto]{x^\dagger - x^\delta}_\cX$ generally being unbounded with respect to choice of $n$. 
This necessitates that a regularization method be used.  

We briefly review some basic methods relevant to the work in this paper.
Many of these methods can be expressed in terms of the introduction of a 
\emph{filter factor} $f(\sigma_i)$ into the expression for $x^\delta$ which 
have the desired damping effect on the smaller singular values, with 
\begin{align*}
	x^\dagger\approx \sum_{i=0}^\infty f(\sigma_i)\dfrac{\left( y^\delta, z_i \right)_\cY}{\sigma_i} w_i.
\end{align*}

The most basic of these methods is the truncated singular value decomposition (\tsvd), which can be expressed in terms
of its filter function 
\begin{align*}
	f_k(x) 
	=
	\begin{cases}
		1\qquad & \mbox{if}\ x \geq \sigma_k\\
		0 \qquad & \mbox{if}\ x< \sigma_k
	\end{cases}
\end{align*}
leading to the approximation $x^\delta_{\mathrm{T},k}
	=
	\sum_{i=0}^k \dfrac{\left( y^\delta, z_i
\right)_\cY}{\sigma_i} w_i$.   This method can be quite effective in some circumstances, but it 
requires that we possess enough of the singular system of $A$ to build the \tsvd solution, which is sometimes a strong assumption. 

Instead we avail ourselves of methods that apply the spectral filtering implicitly.  Consider
the standard Tikhonov regularization, which balances fidelity to the data against overfitting to the data
by penalizing the norm of candidate solutions
\begin{align*}
	x_c^\delta
	=
	\argmin_{x\in\cX}\left\lbrace \norm[auto]{y^\delta - Ax}_\cY^2 + c\norm[auto]{x}_\cX^2 \right\rbrace.
\end{align*}
It is known that solving this minimization for parameter
$c>0$ is equivalent to solving
\begin{align}\label{eqn.Tikhonov-system}
    (A^{\ast}A + c I_{\cX})x_c^\delta=A^{\ast}y^\delta,
\end{align}
where $I_{\cX}\in\cL(\cX)$ is the identity operator for $\cX$.
If $\braces{w_i, z_i, \sigma_i}$ is the singular system of $A$,
we can express the Tikhonov solution as 
\begin{align*}
	x_c^\delta
	=
	\sum_{i=0}^\infty \dfrac{\sigma_i}{\sigma_i^2+c}(y^\delta,z_i)_\cY w_i
	=
	\sum_{i=0}^\infty
	\dfrac{\sigma_i^2}{\sigma_i^2+c}\cdot\dfrac{(y^\delta,z_i)_\cY}{\sigma_i}w_i.
\end{align*}
Thus, the Tikhonov filter function is $f_c(x)
	=
	\dfrac{x^2}{x^2+c}$.
Tikhonov can be quite an effective regularization method, but choosing a
parameter $c$ appropriately for a given noisy right-hand side $y^\delta$ can be
challenging. O'Leary, for example, observed in \cite[Section 2]{O’Leary:2001:1}
that on paper with full knowledge of the noise and of the true solution, we can
quantify a theoretical optimal Tikhonov parameter $c_{opt}$ by using a
Newton-Raphson iteration to find the minimizer of the true error functional
$\norm[auto]{x_c^{\delta} - x}_\cX$.  This is not computable in practice, but
we use it for demonstration purposes with the $\mathtt{shaw\_chebfun()}$
example. For more realistic, large-scale problems, minimizing the Tikhonov
functional for different values of $c$ may be quite expensive, as one seeks to
find a value $c$ satisfying some parameter-choice rule, \eg via the discrepancy
principle \cite[Section 5.1]{EnglHankeNeubauer:1996:1}.  Thus, we consider
iterative regularization approaches which also can be combined with Tikhonov
regularization.

\section{Self-adjoint Lanczos, \cg and, \cgt}
\label{section:Lanczos-CG-CGT}
Consider
the \emph{Landweber iteration}, a gradient descent-type iteration, 
\begin{align*}
	x_{m+1}
	=
	x_m + \alpha A^\ast(y^\delta - Ax_m),
\end{align*} 
where for simplicity we assume
$x_0=0$.  It has been shown that 
\begin{align*}
	x_m
	=
	\alpha\sum_{i=0}^{m-1}\left(I_\cX - \alpha A^\ast A\right)^i A^\ast y^\delta;
\end{align*} 
i.e., it selects its approximation from a Krylov subspace, cf. \eqref{eqn.Krylov}.
The Landweber method is often presented with a fixed step size $\alpha$, chosen to guarantee
certain convergence and regularization properties; see, e.g., \cite{EnglHankeNeubauer:1996:1}.  Alternatively, one can choose $\alpha$ to
minimize the $A^\ast A$-norm of the error over all possible gradient steps.
This method is often called the \emph{steepest descent} method.

\subsection{Self-adjoint Lanczos}
\label{subsection:Lanczos}
The \cg method is often derived and implemented as a modified steepest descent method, or as \lsqr when applied to the normal equations; 
see, e.g., \cite[Chapter 2]{Greenbaum:1997:1} and \cite{PaigeSaunders:1982:1}, respectively. 
However, for our purposes, it is
helpful to express it as a Krylov subspace method built on a self-adjoint
Lanczos iteration.  Krylov subspace methods are workhorse methods for the treatment of well- and 
ill-posed problems.  Much work has been done to extend their analysis to the infinite-dimensional,
ill-posed problems setting; see, \eg \cite{AlqahtaniMachReichel:2022:1,CarusoMichelangeliNovati:2019:1,CarusoNovati:2019:1,Novati:2017:1,Novati:2017:2}.

Consider building a basis for the Krylov 
subspace
\begin{align}
	\cK_m (A^{\ast}A,A^{\ast}y)
	=
	\Span
	\braces{A^{\ast}y,A^{\ast}A(A^{\ast}y),\prn{A^{\ast}A}^{2}(A^{\ast}y),\ldots, \prn{A^{\ast}A}^{m-1}(A^{\ast}y)}\subset\cX.
	\label{eqn.Krylov}
\end{align}
The self-adjoint Lanczos iteration is a short-recurrence method for 
generating an orthonormal basis for $\cK_m (A^{\ast}A,A^{\ast}y)$.
This basis is generated iteratively, with the first vector being 
$v_1=A^{\ast}y/\norm[auto]{A^{\ast}y}_{\cX}$.  Thereafter, given that we have
generated $\braces{v_1,v_2,\ldots,v_i}$, an $\cX$-orthonormal 
basis for $\cK_i (A^{\ast}A,A^{\ast}y)$, we generate the next 
basis vector by orthogonalizing $A^{\ast}A v_i$ against 
$\braces{v_1,v_2,\ldots,v_i}$ and then normalizing.  It can be 
shown that $A^{\ast}A$ being self-adjoint implies that $A^{\ast}A v_i$
is naturally orthogonal to all basis vectors except $v_{i-1}$ 
and $v_i$, i.e., we have the three-term recurrence,
\begin{align}
    	A^{\ast}A v_i
	=
	b_{i+1}v_{i+1} + a_i v_i + b_i v_{i-1},
	\label{eqn.three-term}
\end{align}
where $a_i = \inner[auto]{A^\ast A v_i}{v_i}_\cX$ and $b_i =
\inner[auto]{A^\ast A v_i}{v_{i-1}}_\cX$, which can be used to obtain the next
basis vector $v_{i+1}$. We can store these Lanczos vectors as \dquotes{columns}
of $V_m \in\cL(\R^{m},\cX)$, which is the continuous analog to a tall narrow
matrix. The operator $V_m$ acts on $\R^{m}$ via \emph{linear combination},
where for $\by=(\xi_i)_{i=1}^{m}\in\R^{m}$ we have $V_m: \by\mapsto
\sum_{i=1}^{m} \xi_i v_i$.  It follows that $V_m$ acts similarly on compatible
matrices; i.e., it also holds $V_m \in\cL(\R^{m\times k},\cX^k)$ with $V_m$
simply acting on each column of the matrix.  Denoting $b_1 = \norm[auto]{A^\ast
y^\delta}_\cX$, the Lanczos orthogonalization coefficients are stored in a
tri-diagonal matrix $\underline{\bT_m }\in\R^{(m+1)\times m}$ such that
\begin{align}
	V_m 
	=
	\begin{bmatrix} 
		v_1 & v_2 & \cdots & v_m  
	\end{bmatrix}
	\qquad \mand\qquad
	\underline{\bT_m }
	=
	\begin{bmatrix} 
		a_1      & b_2      &        &         & 	 \\
		b_2      & a_2      & b_3    &         &         \\
			 & b_3      & a_3    & \ddots  &         \\
			 &          & \ddots & \ddots  & b_m    \\
			 &          &        & b_m    & a_m    \\
		         &          &        &         & b_{m+1} 
	\end{bmatrix}.
	\label{eqn:Lanczos-basis-matrix}
\end{align}
With these, we can encapsulate the whole self-adjoint Lanczos process 
via the \emph{Lanczos relation}
\begin{align}
	A^{\ast}AV_m 
	=
	V_{m+1} \underline{\bT_m }
	=
	V_m \bT_m  + b_{m+1}v_{m+1}\be_m^T,
	\label{eqn.lanczos-relation}
\end{align}
where $\bT_m \in\R^{m\times m}$ is the square, tri-diagonal matrix formed by
taking the first $m$ rows of $\underline{\bT_m }$, and $\be_m \in\R^{m}$ is the
$m$th Euclidean basis vector.  We note that 
\linebreak
the symmetry of $\bT_m$ implies that $b_{i+1}$ can be obtained when 
\linebreak 
normalizing $A^\ast A v_i - a_i v_i - b_i v_{i-1}$ to obtain $v_{i+1}$ and then
re-used in the subsequent iteration rather than computing
$b_{i+1}=\inner[auto]{A^\ast A v_{i+1}}{v_{i}}_\cX$.

\subsection{Conjugate Gradients via Lanczos}
\label{subsection:CG-Lanczos}
It can be shown that obtaining the $m$th \cg iterate $x_m $ via the 
modified steepest-descent formulation is equivalent to solving the
tri-diagonal linear system
\begin{align}\label{eqn.CG_step}
	\bT_m \by_m 
	=
	b_1 \be_1 ,
\end{align}
where $\be_1 \in\R^{m}$ is the first Euclidean basis vector,
and setting 
	$
	x_m 
	=
	V_m \by_m 
	$.
This is shown, e.g., in \cite[Chapter 6.7]{Saad:2003:1}, to follow from the fact that the $m$th \cg residual satisfies the Galerkin conditions
\begin{align}
	A^\ast y^\delta - A^\ast A x_m 
	\perp_\cX 
	\cK_m(A^\ast A, A^\ast y^\delta).
	\label{eqn:CG-Galerkin}
\end{align}
In \cite[Chapter 6.7]{Saad:2003:1}, it is further laid out how one
transforms the Lanczos-based derivation into the more familiar formulation that uses
modified gradients as search directions.  It is beyond the scope of this paper to 
rehash the full derivation.  The key insight is that $\bT_m$ always admits an 
unpivoted LU-factorization of the form 
	$\bT_m
	=
	\bL_m\bD_m\bL_m^T$, and $P_m=V_m\bL_m$
has as its \dquotes{columns} these modified gradients. The \cgne approach paired with a stopping rule based on the discrepancy
principle (see, e.g., \cite{ChungGazzola:2024:1,Hanke:2017:1}) is an effective manner in which to truncate iterations before the
\cgne solution is corrupted by noise amplification, which was shown in \cite{Hanke:2017:1} to be a formal regularization under some
general assumptions.  We compare the performance of this approach with optimal Tikhonov in \Cref{fig:shaw-CGdiscrep-and-Tikh}.
\begin{figure}
	\includegraphics[scale=0.7]{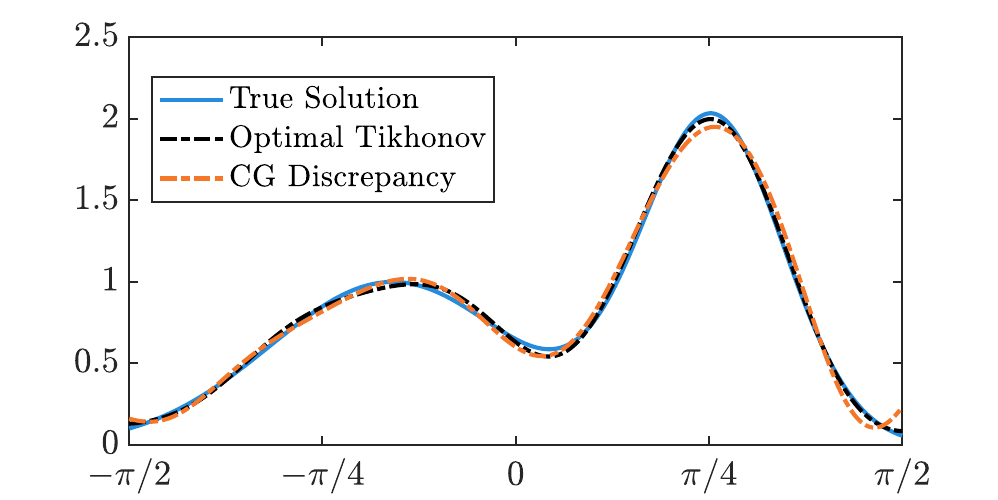}
	\caption{For the $\mathtt{shaw(400)}$ test problem with relative noise-level $10^{-4}$, a comparison of the true solution
	to the regularized solutions produced by Tikhonov with theoretical optimal parameter $c_{opt}$ and \cg stopped using the
	noise-based discrepancy principle.\label{fig:shaw-CGdiscrep-and-Tikh}}
\end{figure}

\subsection{Hybrid \cgt via Lanczos}
\label{subsection:cg-tikhonov}
As $c>0$, the operator in \eqref{eqn.Tikhonov-system} is also self-adjoint and positive-definite
and \eqref{eqn.Tikhonov-system} can thus be solved using \cg.  It has been proven that Krylov
subspaces are identity-shift invariant, i.e., 
\begin{align}
	\cK_m (A^{\ast}A,A^{\ast}y)
	=
	\cK_m (A^{\ast}A + c I_{\cX},A^{\ast}y)
\end{align}
as long as the starting vector remains fixed (or they are scalar 
multiples of one another).  This has been used to derived efficient
\cg-based Tikhonov solvers; see, e.g., \cite{FrommerMaass:1999:1}.  

From \eqref{eqn.lanczos-relation}, one sees
we have the shifted Lanczos relation
\begin{align}
	(A^{\ast}A + c I_{\cX})V_m 
	=
	V_m (\bT_m  + c\bI) + \beta_{m+1}v_{m+1}e_m ^{\ast}.
	\label{eqn.Am-shift-relation}
\end{align}
Thus, applying $j$ iterations of \cg to \eqref{eqn.Tikhonov-system} to
obtain $x_m ^{(c)}$ is equivalent to solving the linear system
\begin{align}\label{eqn.CGT_step}
    (\bT_m  + c\bI)\by_m ^{(c)}
	=
	\norm[auto]{A^{\ast}y}_{\cX}\be_1 ,
\end{align}
and setting $x_m ^{(c)} = V_m \by_m ^{(c)}$. This can be seen from applying
\eqref{eqn:CG-Galerkin} to define $x_m ^{(c)} $, inserting
\eqref{eqn.Am-shift-relation}, and then simplifying. This allows us to relate
the \cg iterate $x_m $ and $x_m ^{(c)} $ using relations depending on $c$.

\begin{remark}
	We note that this approach to \cgt, whereby one runs a Krylov subspace
	iteration and applies the Tikhonov regularization on a projected
	representation of the problem \cite[Section
	5.2]{EnglHankeNeubauer:1996:1} on the subspace is an early example of a
	so-called \emph{hybrid} approach.  See, \eg,
	\cite{ChungNagyOLeary:2007:1, LewisReichel:2009:1,
	KilmerO’Leary.:2001:1} for other examples of hybrid methods as well as
	\cite{Novati:2017:1} for an analysis of hybrid the Arnoldi-Tikhonov
	approach in infinite dimensions.
\end{remark}

One observation made in \cite{FrommerMaass:1999:1} is that although the normal equations and 
Tikhonov problems generate the same shift-invariant Krylov subspace, the
modified gradient search directions do not satisfy any such invariance or a simple
relationship.  One must compute new directions for each choice of parameter $c$.

\subsection{Filter functions of \cg and \cgt}
\label{subsection:filters-CG-CGTikh}
As discussed in \Cref{subsection:CG-Lanczos}, \cg applied to the normal equations produces approximations of the form
$p_m(A^\ast A)A^\ast y^\delta$ where $p_m(z)$ is a polynomial of degree at most $m-1$.  This polynomial has some
interesting properties.  It can be expressed as the polynomial which minimizes the residual norm functional
\begin{align*}
	p_m(z)
	=
	\argmin_{p\in\Pi_{m-1}}\norm[auto]{
		y^\delta 
		- 
		Ap(A^\ast A)A^\ast y^\delta
	}_\cY,
\end{align*}
where $\Pi_i$ is the space of polynomials of degree no greater than $i$.
We can again use \eqref{eqn.svd-A-mapping} and \eqref{eqn.svd-Ane-mapping} to express the \cg iterates in a
way that lets us expose the filter functions of the \cg iterates,
\begin{align*}
	x_m 
	&= 
	p_m(A^\ast A)A^\ast y^\delta
	\\
	&=
	\sum_{i=0}^\infty \sigma_i p_m(\sigma_i^2)(y^\delta, u_i)_\cY v_i
	\\
	&=
	\sum_{i=0}^\infty \sigma_i^2 p_m(\sigma_i^2)\dfrac{(y^\delta, u_i)_\cY}{\sigma_i} v_i.
\end{align*}
Thus the $m$th filter function for \cgne is $f_m(x)
	=
	x p_m(x)$; see, e.g., \cite{HankeNagyPlemmons:1993:1,VanDerSluisVanDerVorst:1990:1}.  We have quite a bit of information
	about this polynomial filter function.  The associated normal equations residual polynomial for $A^\ast(y^\delta - A x_m)$
	is $r_m(x)
	=
	1 - x p_m(x)
	=
	1-f_m(x)$.  It is a well-known result from \cite{GoossensRoose:1999:1} that the roots of the
residual polynomial are the eigenvalues $\braces{\eta_i^{(m)}}_{i=1}^m$ of the tridiagonal matrix $\bT_m$, i.e.,
$r_m(\eta_i^{(m)})=0$ \cite{GoossensRoose:1999:1}.  This implies that $f_m(\eta_i^{(m)})
	=
	1$.  These eigenvalues are called \emph{Ritz values},
and they approximate eigenvalues of $A^\ast A$.  Thus, if $\eta_i^{(m)}$ approximates $\sigma_k^2$ for some $k$, then 
$f_m(\sigma_k^2) \approx 1$, and the associated component of $x^\dagger$ will not be filtered out.  This also
illuminates a mechanism by which too many iterations of \cg will begin to overfit to the noisy $y^\delta$.  As the
Krylov subspace gets larger, the eigenvalues of $\bT_m$ begin to approximate $\sigma_k^2$ for increasingly large
values of $k$, meaning noisy components of $x^\dagger$ begin to not be filtered out.

\subsection{Filter functions of \cgt}
\label{subsection:CGTikh-filter}
We consider applying \cg to the Tikhonov linear system \eqref{eqn.Tikhonov-system}.  Since Krylov subspaces are invariant
with respect to a shift by a scalar multiple of the identity, we can express the \cgt iterate using the same type
of polynomial filter analysis.  The $j$th \cgt iterate can be expressed as $x_m^{(c)}
	=
	p_m(A^\ast A + c
	I_\cX)A^\ast b$.  
Following the same steps as with \cg, we use \eqref{eqn.svd-A-mapping} and \eqref{eqn.svd-Ane-mapping} to
express the iterate as 
\begin{align*}
	x_m^{(c)}
	=
	\sum_{i=0}^\infty \sigma_i^2 p_m^{(c)}(\sigma_i^2 + c)\dfrac{(y^\delta,u_i)_\cY}{\sigma_i}v_i.
\end{align*}
Thus, the \cgt filter function is 
\begin{align}
	f_m^{(c)}(x)
	=
	xp_m^{(c)}(x+c)
	.
	\label{eqn:cgt-filter-function}
\end{align}
From \cite{GoossensRoose:1999:1}, the roots of the \cgt residual polynomial are the eigenvalues 
of the shifted tridiagonal $\bT_m+c\bI$; \ie the shifted Ritz values $\braces{\eta_i^{(m)}+c}_{i=1}^{m}$
are the roots of the \cgt residual polynomial $r_m^{(c)}(t)
	=
	1 - t p_m^{(c)}(t)$. With the variable transform
$t\rightarrow x+c$, we have that the unshifted Ritz values $\braces{\eta_i^{(m)}}_{i=1}^m$
are roots of the shifted residual polynomial; \ie 
\begin{align*}
	r_m^{(c)}(x+c)
	=
	1 - (x+c)p_m^{(c)}\prn{x+c}, 
\end{align*}
	which means that 
\begin{align}
	(\eta_i^{(m)}+c)p_m^{(c)}\prn{\eta_i^{(m)}+c}
	=
	1
	. 
	\label{eqn:shift-resid-poly-root-consequence}
\end{align}
Combining \eqref{eqn:cgt-filter-function} with
\eqref{eqn:shift-resid-poly-root-consequence}, it follows that
$f_m^{(c)}\prn{\eta_i^{(m)}}
	=
	\dfrac{\eta_i^{(m)}}{\eta_i^{(m)} + c}$.
	Thus, if $\eta_i^{(m)}$ approximates $\sigma_k^2$ for some $k$, 
then $f_m^{(c)}(\sigma_i^2)$ will begin to approximate the true Tikhonov filter for that singular value.  As the                                      
Krylov subspace gets larger, the eigenvalues of $\bT_m$ begin to approximate $\sigma_k^2$ for increasingly large                                 
values of $k$, meaning the components associated to well-approximated small singular values will begin to be filtered out.

We observe that expressing the iterates of \cg and \cgt in this manner does not allow us to relate the approximations of the two
methods in a straightforward manner, even though they are both expressed in terms of the filtration of same approximate
eigenvalues in two different ways.  

We demonstrate that we can take advantage of the Krylov subspace shift
invariance to more directly relate the two methods.  To do this, we must
characterize the relationship between inverses of linear shifts of a
symmetric tridiagonal matrix.  We take advantage of well-developed
characterization of tridiagonal matrices and their inverses in the
literature.

\section{On the inverse of a shifted tri-diagonal}
\label{section:tri-diag-inverse}

An explicit representation of the entries of the inverse of a 
tri-diagonal matrix has been calculated.  This was originally done
in \cite{Usmani:1994:1} with a more compact formulation
being shown in \cite{DaFonseca:2007:1}.  We adapt these
to the case of our symmetric tri-diagonal $\bT_m $.
\begin{proposition}[adapted from \cite{DaFonseca:2007:1}]
Denoting by $(\bT_m ^{-1})_{ij}\in\R$ the $(ij)$th entry of  
$\bT_m ^{-1}$ with diagonal entries $a_i$ and super-/sub-diagonal entries $b_i$ defined as in the Lanczos recurrence \eqref{eqn:Lanczos-basis-matrix}, we can write
\begin{align}\label{eqn.Tm-inverse}
    (\bT_m ^{-1})_{ij}
	=
    \dfrac{(-1)^{i+j}}{\theta_m }
    \begin{cases} 
         b_{i+1}\cdots b_m \theta_{i-1}\phi_{m+1}^{(m)}\qquad &\mbox{if}\ i< j\\
        \theta_{i-1}\phi_{m+1}^{(m)}\qquad &\mbox{if}\ i= j\\
         b_{m+1}\cdots b_i \theta_{j-1}\phi_{i+1}^{(m)}\qquad &\mbox{if}\ i> j
    \end{cases},
\end{align}
where $\theta_\ell$ and $\phi_\ell^{(m)}$ satisfy the recurrence relations
\begin{align}\label{eqn.thetas}
    \theta_\ell
	=
	a_\ell\theta_{\ell-1} -  b_\ell^{2}\theta_{\ell-2}\qquad\mbox{for}\ \ell=2,3,\ldots,m
\end{align}
with the initial values $\theta_0 =1$ and $\theta_1 =a_1 $, and
\begin{align}\label{eqn.phis}
    \phi_\ell^{(m)}
	=
	a_\ell\phi_{\ell+1}^{(m)} -  b_{\ell+1}^{2}\phi_{\ell+2}^{(m)}\qquad\mbox{for}\ \ell=m-1,m-2,\ldots, 1
\end{align}
with the initial values $\phi_{m+1}^{(m)}=1$ and $\phi_m ^{(m)}=a_m $.
Furthermore, $\theta_i$ is the determinant of $\bT_i$.
\end{proposition}
It should be noted that there are alternative ways to express this inverse, and these are explored in the book
\cite[Chapter 2.5]{Meurant:2025:1}. 
Using (\ref{eqn.Tm-inverse}--\ref{eqn.phis}), we can develop an expression of $x_m$ in the basis of Lanczos vectors. 
\begin{corollary}
The $m$-th \cgne-iterate admits the following representation in the Lanczos basis, 
\begin{align}
	x_m
	=
	\dfrac{1}{\det \bT_m}
	\sum_{i=1}^m
	(-1)^{i+1}
	b_1\cdots b_i
	\phi_{i+1}^{(m)}
	v_i.
	\label{eq:CGNE-Lanczos-basis}
\end{align}
\end{corollary}
\begin{proof}
It follows from \eqref{eqn.CG_step} that
\begin{align*}
	x_m 
	=& 
	V_m\prn{T_m^{-1}\prn{b_1\be_1}}
	\\
	=&
	b_1 V_m
	\prn{T_m^{-1}}_{:,1}
\end{align*}
where from \eqref{eqn.Tm-inverse} the first column of the inverse can be written as
\begin{align*}
	\prn{T_m^{-1}}_{:,1}
	=&
	\dfrac
	{1}
	{\det T_m}
	\begin{cases}
		\phi_2^{(m)} 
		& 
		i = 1
		\\
		(-1)^{i+1}b_2\cdots b_i\phi_{i+1}^{(m)} 
		&
		i > 1
	\end{cases}.
\end{align*}
\end{proof}
Consider now that \eqref{eqn.Tm-inverse} holds also for the diagonally shifted symmetric tri-diagonal matrix $\bT_m  + c\bI$,
	where diagonal entry $a_\ell$ is replaced with $a_\ell + c$ and the entries $ b_\ell$ remain unchanged.  We can use
	this to characterize the entries of $ \prn{\bT_m +c\bI}^{-1}$ in terms of entries of $\bT_m ^{-1}$.
	Applying \eqref{eqn.Tm-inverse} for $\bT_m  + c\bI$, we denote the auxiliary quantities, defined respectively in
	\eqref{eqn.thetas} and \eqref{eqn.phis}, for the shifted matrix by $\theta_\ell(c)$ and $\phi_\ell^{(m)}(c)$.
\begin{lemma}\label{lemma.Tsig-inverse}
     It holds for all $\ell$ that $ \theta_\ell(c) = \theta_\ell + g_\ell(c)$ and  
     $\phi_i^{(m)}(c) = \phi_i ^{(m)} + h_i ^{(m)}(c)$ where $g_\ell$ and $h_\ell$ satisfy the recurrences
\begin{align}
    	g_\ell(c)
	=
	\paren{(}{)}{a_\ell + c} g_{\ell-1}(c) -  b_\ell^{2}g_{\ell - 2}(c) + c \theta_{\ell-1},
	\label{eqn.g-recur}
\end{align}
where we define $g_0 (c)=0$ and $g_1 (c)=c$, and
\begin{align}
	h_\ell^{(m)}(c)
	= 
	\prn{a_\ell + c} h_{\ell+1}^{(m)}(c) 
	-  
	b_{\ell+1}^{2}h_{\ell+2}^{(m)}(c) 
	+ 
	c \phi_{\ell+1}^{(m)},
	\label{eqn.h-recur}
\end{align}
where we define $h_{m+1}^{(m)}(c)=0$ and $h_m ^{(m)}(c) = c$.
\end{lemma}
\begin{proof}
    From \eqref{eqn.Tm-inverse}, we already have a representation for the entries of $\prn{\bT_m  + c\bI}^{-1}$. We observe that,
	since the diagonal shift does not affect the off-diagonal entries, the shift $c$ only affects 
	$\theta_\ell(c)$ and $\phi_\ell^{(m)}(c)$.
	Note that these satisfy the same recurrences as in the base case but with different initial conditions, namely
    \begin{align}
        \theta_0(c)
	=
	\theta_0 ;\  \theta_1(c)
	=
	a_1  + c
	=
	\theta_1  + c;\  \phi_{m+1}^{(m)}(c)
	=
	\phi_{m+1}^{(m)};\  \phi_m^{(m)}(c)
	=
	a_m  + c
	=
	\phi_m ^{(m)} + c.
    \end{align}
    If we denote $g_0 (c)=0$, $g_1 =c$, $h_{m+1}^{(m)}(c)=0$,
    and $h_m ^{(m)}(c)=c$, then we have shown the lemma holds
	for the base case of the proposed recurrences, \eqref{eqn.g-recur} and \eqref{eqn.h-recur}, for $g_\ell$ and $h_\ell^{(m)}$.
    We simply need to prove it holds for an induction step.  
    Suppose for $0\leq i \leq \ell-1$, we have 
    $\theta_i(c)
	=
	\theta_i  + g_i (c)$.  Then we know
    from the recurrence in \eqref{eqn.Tm-inverse} that 
    \begin{align}
        \theta_\ell(c)
	=
	\prn{a_\ell+c}\prn{\theta_{\ell-1} + g_{\ell-1}(c)} -  b_\ell^{2}\prn{\theta_{\ell -2} + g_{\ell-2}(c)}.
    \end{align}
    Multiplying this out, we obtain $ \theta_\ell(c) = \theta_\ell + g_\ell(c)$, where $g_\ell(c)$ is indeed the same
	as the recurrence formula stated in the lemma. Similarly, let us assume that for $\ell+1\leq i\leq m+1$, it holds that
	$\phi_i^{(m)}(c) = \phi_i ^{(m)} + h_i ^{(m)}(c)$.  Then we know from the recurrence in \eqref{eqn.Tm-inverse} that
    \begin{align}
        \phi_\ell^{(m)}(c)
	=
	\prn{a_\ell + c}\prn{\phi_{\ell+1}^{(m)} + h_{\ell+1}^{(m)}(c) }-  b_{\ell+1}^{2}\prn{\phi_{\ell+2}^{(m)} + h_{\ell + 2}^{(m)}	(c)}.
    \end{align}
    The proof proceeds similarly as for $\theta_\ell(c)$ in
    that we insert the expressions from the assumption at the 
    induction step and simplify.  This completes the proof.
\end{proof}
\begin{corollary}
	We have the relationship between determinants 
	\begin{align*}
		\det(\bT_m + c\bI) = \theta_m(c) = \theta_m + g_m(c) = \det \bT_m + g_m(c).
	\end{align*}
\end{corollary}
Beginning by using \eqref{eq:CGNE-Lanczos-basis} to express the \cgt-iterate in terms of the Lanczos basis, we derive an expression
relating corresponding \cgne and \cgt iterates.
\begin{corollary}
	The $m$-th \cgt-iterate admits the following representation in the Lanczos basis, 
\begin{align}
	x_m^{(c)}
	=&
	\dfrac{1}{\det\prn{\bT_m + c\bI}}
	\sum_{i=1}^m
	(-1)^{i+1}
	b_1\cdots b_i
	\prn{\phi_{i+1}^{(m)} + h_{i+1}^{(m)}(c)}
	v_i
	\nonumber
	\\
	=&
	\frac{\det\bT_m}{\det\prn{\bT_m + c\bI}} x_m +\frac{1}{\det\prn{\bT_m + c\bI}}\prn{\sum_{i=1}^{m}(-1)^{i+1} b_1 \cdots b_i h_{i+1}^{(m)}(c)v_i}
	\label{eq:CGT-Lanczos-basis}
\end{align}

\end{corollary}
\begin{figure}
	\includegraphics[scale=0.5]{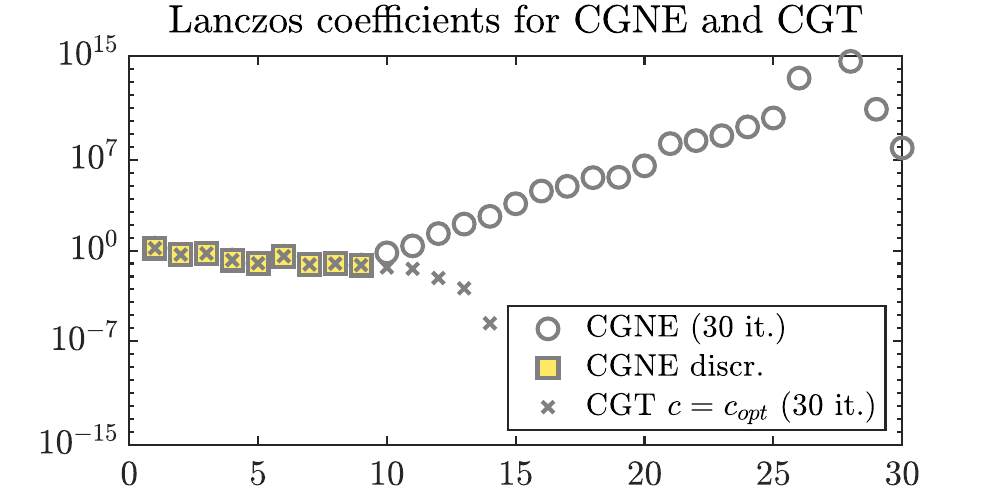}
	\caption{\label{figure:shaw-Lanczos-coeffs} The \emph{absolute values} of the Lanczos vector coefficients of \cgne(30),
	\cgne stopped using the discrepancy principle, and \cgt(30) using $c=c_{opt}$ from \cite{O’Leary:2001:1}. It should be noted
	that the coefficients from discrepancy principle \cg are indeed close to the first seven coefficients of the \cgt(30)
	iterate.}
\end{figure}
Important for understanding the relationship between \cg and \cgt as the regularization parameter varies is the
following characterization of $g_\ell(c)$ and $h_\ell^{(m)}(c)$.
\begin{corollary}\label{corollary.f-g-poly}
Both $g_\ell(c)$ and $h_\ell^{(m)}(c)$ are monic polynomials in $c$.  The function $g_\ell(c)$
	is a monic polynomial of degree $\ell$, and $h_\ell^{(m)}(c)$ is a monic polynomial of degree $m-\ell+1$.
\end{corollary}
\begin{proof}
In both cases, we can prove the characterization by induction.  For $g_\ell(c)$, the base cases for $\ell=0$ and
$\ell
	=
	1$ are defined in \Cref{lemma.Tsig-inverse}.  For the general case, we assume that the statement holds for
$\ell-1$ and $\ell-2$.  From \Cref{lemma.Tsig-inverse}, it follows that $g_\ell(c)
	=
	c g_{\ell-1}(c) +
J_{\ell-1}(c)$, where $\deg J_{\ell-}(c)
	=
	\ell-1$. Thus, the statement holds. 

Similarly, for $h_\ell^{(m)}(c)$, the base cases for $\ell=m+1$ and $\ell=m$ are also defined in
\Cref{lemma.Tsig-inverse}.  For the general case, we similarly assume that the statement holds for $\ell+1$ and
$\ell+2$.  From \Cref{lemma.Tsig-inverse}, it follows that $h_\ell^{(m)}(c)
	=
	c h_{\ell+1}^{(m)}(c) +
K_{\ell+1}^{(m)}(c)$ where $\deg K_{\ell+1}^{(m)}(c)
	=
	m-\ell$. 
\end{proof}
\begin{remark}
	As an aside, we note that these polynomials \emph{almost} satisfy a three-term recurrence, but do not because the recurrence
	always involves a shift in the linear term.  There is thus is no inner product for which $\paren[auto]{\{}{\}}{g_\ell}$ or
	$\paren[auto]{\{}{\}}{h_\ell^{(m)}}$ are families of orthonormal polynomials.
\end{remark}
The following corollary follows from \eqref{eqn.CGT_step},
but we prove it instead using the Lanczos filtration quantities,
which gives a better qualitative sense of the nature of the convergence.
\begin{corollary}
	It follows that $\lim_{c\rightarrow\infty}x_m^{(c)}
	=
	0$.
\end{corollary}
\begin{proof}
We prove this by studying the individual terms in \eqref{eq:CGT-Lanczos-basis}.
We have that $\lim_{c\rightarrow\infty}g_m(c)=\infty$.  Since $\theta_m$ does
not depend on $c$, it follows that\linebreak
\begin{math}
	\lim_{c\rightarrow\infty}\dfrac{\theta_m}{\theta_m + g_m(c)}
	=
	0 
	.
\end{math}
Furthermore, we observe that
$\lim_{c\rightarrow\infty}\dfrac{h_\ell^{(m)}(c)}{\theta_m + g_m(c)}=0$, since
$\deg h_\ell^{(m)} < \deg g_m$ for all $\ell$, $2\leq\ell\leq m+1$. This
completes the proof.  \end{proof} This result is not surprising, since
classical Tikhonov regularization induces a smoothing effect on the solution,
and the strength of that smoothing increases with the size of the
regularization parameter.

To further understand and contextualize the relationship of \cgne and \cgt iterates, we transition to developing some additional
analysis of these iterations.  We build on some recent work analyzing these methods in infinite dimensions, but working in the
more common, on-paper equivalent but numerically more stable Golub-Kahan Bidiagonalization-based \lsqr\ approach
\cite{PaigeSaunders:1982:1}.  This work builds on that in \cite{AlqahtaniRamlauReichel:2022:1,CarusoNovati:2019:1}. We note that an analysis of the hybrid Arnoldi-Tikhonov approach has also been 
undertaken in the infinite-dimensional setting \cite{Novati:2017:1}.

\section{Golub-Kahan Bidiagonalization and Lanczos in infinite dimensions}
\label{section:GKB-Lanczos-inf-dim}
\cgne can be equivalently formulated by studying the Golub-Kahan bidiagonalization (\gkb)
method with respect to $A$.  This leads to the method called \lsqr, which is equivalent in exact arithmetic to \cgne but known to be
numerically more stable.  The basic observation leading to the \gkb is that for any $A$, we can build companion orthonormal bases for the Krylov 
subspaces $\cK_m (AA^{\ast},y^\delta)$ and $\cK_m (A^{\ast}A,A^{\ast}y^{\delta})$.  From this we build the short recurrence method called \lsqr.

Consider $U_m  = \begin{bmatrix} u_1  & u_2  & \cdots & u_m \end{bmatrix}\in\cL(\R^m,\cY)$ and $V_m  = \begin{bmatrix} v_1 & v_2 & \cdots & v_m
\end{bmatrix}\in\cL(\R^m,\cX)$
which have orthonormal columns and are defined recursively by
\begin{align}
	u_1  = y^{\delta}/\beta_0 ;\ \ \beta_0  
	&= 
	\norm[auto]{y};\ \ \alpha_1 v_1 = A^{\ast}u_1 
	\nonumber
	\\
	Av_i = \alpha_i u_i  + \beta_{i+1}u_{i+1};
	\ &\ 
	A^{\ast}u_{i+1} = \beta_{i+1}v_i + \alpha_{i+1}v_{i+1}.
	\label{eqn:GK-3-term-recur}
\end{align}
If we define 
\begin{align*}
	\underline{\bB_m}  
	=
	\begin{bmatrix}
		 \alpha_1  &  &  &  \\ 
		 \beta_2  & \alpha_2  &  &  \\ 
		  & \beta_3  & \ddots &  \\ 
		  &  & \ddots & \alpha_m  \\ 
		  &  &  & \beta_{m+1}
		 \end{bmatrix} 	 
		 \in 
		 \R^{(m+1)\times m},
\end{align*}
then the bases satisfy the Golub-Kahan relations
\begin{align}
	AV_m  
	&= 
	U_{m+1}\underline{\bB_m} 
	\qquad\mand\qquad
	A^{\ast}U_m  
	= 
	V_m \bB_m ^{\ast}.	
	 \label{eqn.GKB-relation}
\end{align}
From this it follows that 
\begin{align*}
	A^{\ast}AV_m 
	&=
	A^{\ast}U_{m+1}\underline{\bB_m} 
	\\
	&=
	V_{m+1}\bB_{m+1}^{\ast}\underline{\bB_m} 
	\\
	&=
	V_{m+1}
	\begin{bmatrix}
		\alpha_1^2+\beta_2^2 & \beta_2\alpha_2      & 		           & 		       &
		\\
		\beta_2\alpha_2      & \alpha_2^2+\beta_3^2 & \beta_3\alpha_3      & 		       &
		\\
		                     & \beta_3\alpha_3      & \alpha_3^2+\beta_4^2 & \ddots            &
		\\
		                     &                      & \ddots	           & \ddots            & \beta_m\alpha_m
		\\
				     &                      & 		           & \beta_m\alpha_m   & \alpha_m^2 + \beta_{m+1}^2
		\\
				     &                      & 		           &                   & \alpha_{m+1}\beta_{m+1}
	\end{bmatrix}. 
\end{align*}
Since both methods produce the same orthonormal basis for $\cK_m(A^\ast A, A^\ast y^\delta)$ and the Lanczos decomposition is 
unique, it follows that $\underline{\bT_m}=\bB_{m+1}^{\ast}\underline{\bB_m}$.  This means that $a_i = \alpha_i^2 + \beta_{i+1}^2$,
and $b_i = \beta_i\alpha_i$.
In the \lsqr implementation, computing the same approximation as the \cgne approximation can be formulated as solving the system
\begin{align*}
	\underline{\bB_m} ^{\ast}\underline{\bB_m} \by_m  = \underline{\bB_m} ^{\ast}(\beta_0 \be_1 ),
\end{align*}
and setting $x_m  = V_m \by_m $.
Similarly the \cgt can also be formulated in the \gkb setting as solving
\begin{align*}
	(\underline{\bB_m} ^{\ast}\underline{\bB_m}  + \lambda^{2}\bI)\by_m ^{(\lambda^{2})} = \underline{\bB_m} ^{\ast}(\beta_0 \be_1 ),
\end{align*}
and setting $x_m ^{(\lambda^{2})} = V_m \by_m ^{(\lambda^{2})}$.

\subsection{\gkb in the infinite dimensional setting}
In \cite{AlqahtaniRamlauReichel:2022:1}, the authors analyze the \gkb-based
\cg/\cgt in the infinite-dimensional inverse problems context in order to prove
some results concerning noise-based convergence results.  The \gkb and \lsqr
algorithms are presented using the techniques and language developed in
\cite{CarusoNovati:2019:1}. In this section, we summarize some of the results
proved in that paper.  We further prove some additional results using the same
techniques which are useful when relating \cg and \cgt.  We note that the
authors consider some special cases that we do not necessarily use in our
setting.  Namely, the authors consider a range-restricted
\cite{CalvettiLewisReichel:2000:1,CalvettiLewisReichel:2001:1} version of the
\gkb-algorithm; i.e., they start the \gkb iteration by defining $u_1 = A A^\ast
y^\delta$ rather than simply using $y^\delta$.  Furthermore, the authors
consider the case that the exact operator $A$ relating the solution to $y$ is
not known and that we instead have an approximation $A_h$ thereof. For
simplicity of exposition, we do not consider these in our setting. 

The authors of \cite{AlqahtaniRamlauReichel:2022:1} do begin with one assumption that we carry over in this paper.  
\begin{assumption}\label{ass:non-terminating-gkb}
	The \gkb iteration introduced at the beginning of this section can be
	run for infinitely many steps. 
\end{assumption}
This assures that we avoid cases for which the method breaks down, such as when dealing with an integral operator induced by a degenerate
kernel.  This leads to an infinite version of \eqref{eqn.GKB-relation}, 
\begin{align}
	AV  
	&= 
	U\bB 
	\qquad\mand\qquad
	A^{\ast}U  
	= 
	V \bB ^{\ast},	
	 \label{eqn.infinite-GKB-relation}
\end{align}
where 
\begin{align*}
	V
	=
	\begin{bmatrix}
		v_1 & v_2 & \cdots & v_m & \cdots
	\end{bmatrix}
	\in &
	\cL\prn{\ell_2, \cX},
	\qquad \mand 
	\\
	U
	=
	\begin{bmatrix}
		u_1 & u_2 & \cdots & u_m & \cdots
	\end{bmatrix}
	\in &
	\cL\prn{\ell_2, \cY}
\end{align*}
are linear operators which act as isometries on $\braces{\kappa_i}\in\ell_2$ \cite{CarusoNovati:2019:1} 
via infinite linear combination $V\braces{\kappa_i} = \sum_{i=1}^\infty \kappa_i v_i$, and  $U\braces{\kappa_i} =
\sum_{i=1}^\infty \kappa_i u_i$, and $\bB\in\cL\prn{\ell_2,\ell_2}$ is the infinite lower bidiagonal matrix 
\begin{align*}
	\bB  
	=
	\begin{bmatrix}
		 \alpha_1  &  &  &  &  \\  
		 \beta_2  & \alpha_2  &  &  &  \\  
		  & \beta_3  & \ddots &  &  \\  
		  &  & \ddots & \alpha_m &  &  \\ 
		  &  &  & \beta_{m+1} & \ddots \\
		  &  &  &  &   \ddots
	\end{bmatrix}, 	 
\end{align*}
and we note that we use Greek letters to denote the entries of bidiagonal $\bB$ to differentiate
them from the entries of tridiagonal $\bT$ from \eqref{eqn:infinite-tridiagonal}, which has 
entries denoted by letters from the Roman alphabet.
A similar setup is also used in the analysis in \cite{CarusoNovati:2019:1}. Following these authors, we denote
\begin{align*}
	\cK_\infty(A^\ast A, A^\ast y^\delta) 
	:= &
	\overline{
		\mathrm{span}\braces{v_1, v_2, \ldots, v_m, \ldots}
	},
	\qquad\mand
	\\
	\cK_\infty(AA^\ast , y^\delta) 
	:= & 
	\overline{
		\mathrm{span}\braces{u_1, u_2, \ldots, u_m, \ldots}
	}.
\end{align*}
It should be noted that this means that \Cref{ass:non-terminating-gkb} also implies that the infinite Lanczos iterations are non-terminating and 
$\cK_\infty(A^\ast A, A^\ast y^\delta)$ and $\cK_\infty(AA^\ast , y^\delta)$ are infinite dimensional.  We have the
following infinite Lanczos relations,
\begin{align}
	A^\ast A V = V \underline{\bT}
	\qquad\mand\qquad
	AA^\ast U = U \dunderline{\bT} 
	\label{eqn:infinite-tridiagonal}
\end{align}
wherein $\underline{\bT} = \bB^\ast\bB$ is the infinite symmetric tridiagonal matrix with diagonal entries 
$a_i = \alpha_i^2 + \beta_{i+1}^2$ and super- and sub-diagonal entries $b_i = \beta_i\alpha_i$, and $\dunderline{\bT}=\bB\bB^\ast$
with diagonal entries $\dunderline{a_i} = \alpha_i^2 + \beta_i^2$ and $\dunderline{b_i} = \alpha_{i-1}\beta_i$, for $i\geq 2$. We note that the entries of $\underline{\bT}$ correspond to those 
of $\bT_m$; \ie $\bT_m$ is the finite $m\times m$ submatrix of $\underline{\bT}$.
The authors of \cite{CarusoNovati:2019:1} are able to prove sufficient conditions for 
infinite dimensional \lsqr to converge to an exact solution, in the case that $y\in\range\prn{A}$.

In \cite{AlqahtaniRamlauReichel:2022:1}, the authors concern themselves with characterizing the entries of $\dunderline{\bT}$, as
this is what is needed in the analysis of the regularizing properties of \lsqr in the setting of that paper.  In this work, we
require a similar characterization of the entries of $\underline{\bT}$.  Fortunately, the same proof approach from
\cite{AlqahtaniRamlauReichel:2022:1} works in this context.
\begin{remark}
	In \cite{AlqahtaniRamlauReichel:2022:1}, the authors observe that
	$\dunderline{\bT}$ is obtained via unitary transformation of an
	infinite-dimensional, compact operator.  However, unlike in finite
	dimensions, it is not the operator $AA^\ast$.  It is clear that
	$\cK_\infty(AA^\ast , y^\delta)\subset \cY$, but the infinite Lanczos
	process may not generate a basis whose closure is $\cY$, with the
	dimension (finite or infinite) of $\cK_\infty(AA^\ast , y^\delta)$
	depending on properties of the operator and the starting vector
	\cite{CarusoNovati:2019:1}.  Instead, it is observed that $AVV^\ast
	A^\ast = U\bB\bB^\ast U^\ast= U\dunderline{\bT} U^\ast$, where
	$\widetilde{P}:=VV^\ast$ is a linear mapping into $\cK_\infty(A^\ast A
	, A^\ast y^\delta)$ \cite{CarusoNovati:2019:1}.  
\end{remark}
We note that since the coefficients $b_i$, $i=1,2,\ldots, m, \ldots$ are norms, it follows from \Cref{ass:non-terminating-gkb} that $b_i> 0$.
Following \cite{AlqahtaniRamlauReichel:2022:1}, we prove the following.
\begin{lemma}\label{lemma:T-offdiagonal-bound}
	Let $\widetilde{Q} := UU^\ast$ be a linear mapping into $\cK_\infty(AA^\ast , y^\delta)$.  Let the
	operator $A^\ast\widetilde{Q} A\in\cL(\cX)$ have the eigensystem $\braces{\lambda_i, s_i}$ with eigenvalue ordering
	$\lambda_{j+1}\geq\lambda_j\geq 0$ for all positive indices $j$.  Then it follows that
	\begin{align*}
		\prod_{i=2}^jb_i = \prod_{i=2}^j \beta_i\alpha_i \leq \norm[auto]{y^\delta}\prod_{i=1}^j \lambda_i.
	\end{align*}
\end{lemma}
\begin{proof}
	This proof follows the same steps as the proof of \cite[Theorem
	1]{AlqahtaniRamlauReichel:2022:1}. We observe that the following
	identity falls out of the second infinite Lanczos relation
	\begin{align}
		A^\ast\widetilde{Q}A
		= &
		\prn{A^\ast U}\prn{A^\ast U}^\ast
		\nonumber
		\\
		=&
		\prn{V\bB^\ast}\prn{V\bB^\ast}
		\nonumber
		\\
		=&
		V\bB^\ast \bB\bV^\ast
		\nonumber
		\\
		=&
		V\underline{\bT}V^\ast
		\label{eq:infinite-tridiagonalization}
	\end{align}
	We write compactly the
	infinite eigen-decomposition $A^\ast\widetilde{Q}A=S\Lambda S^\ast$.  We define an indexed set of monic polynomials
	$p_m(t)=\prod_{i=1}^m(t-\lambda_i)$ having roots at subsets of the eigenvalues of $A^\ast\widetilde{Q}A$.  We apply the
	operator polynomial
	\begin{align}
		\norm[auto]{p_m\prn{A^\ast\widetilde{Q}A}} 
		= 
		\norm[auto]{p_m\prn{\Lambda}} 
		= 
		\sup_{j\geq m+1} \abs{p_m(\lambda_j)}
		\leq
		\abs{p_m(0)} 
		= 
		\prod_{i=1}^m\lambda_i. 
		\label{eq:poly-oper-tridiag}
	\end{align}
	This implies, e.g., that $\norm[auto]{p_m\prn{A^\ast\widetilde{Q}Ay^\delta}} \leq \norm[auto]{y^\delta}\prod_{i=1}^m\lambda_i.$
	Applying this operator polynomial to the starting vector of the Krylov subspace $\cK_\infty(AA^\ast , y^\delta)$ yields
	\begin{align*}
		p_m\prn{A^\ast\widetilde{Q}Ay^\delta}y^\delta
		= &
		U p_m\prn{\underline{\bT}}U^\ast y^\delta
		\\
		= &
		\norm[auto]{y^\delta}U p_m\prn{\underline{\bT}} \be_1.
	\end{align*}
	From this, we may conclude that 
	\begin{align}
		\norm[auto]{p_m\prn{A^\ast\widetilde{Q}A}y^\delta} 
		= &
		\norm[auto]{y^\delta}\norm[auto]{p_m\prn{\underline{\bT}}\be_1}
		\nonumber
		\\
		\geq &
		 \norm[auto]{y^\delta}\norm[auto]{\be_{m+1}^Tp_m\prn{\underline{\bT}}\be_1}.
		 \label{eq:poly-op-lower-bound}
	\end{align}
	The remainder of the proof follows by induction on $m$.  For the case $m=1$, the quantity in the norm from
	\eqref{eq:poly-op-lower-bound} reduces to 
	\begin{align*}
		\be_2^T p_1\prn{\underline{\bT}}\be_1
		= &
		\begin{bmatrix}
			\alpha_2\beta_2 & \alpha_2^2+\beta_3^2-\lambda_1 & \alpha_3\beta_3
		\end{bmatrix}
		\be_1
		\\
		= &
		\alpha_2\beta_2,
	\end{align*}
	and plugging this into  \eqref{eq:poly-op-lower-bound} for the case $m=1$ shows that the result holds in this case.
	Now suppose the result holds for $m= j$.  For the case $m = j+1$, we must calculate
	\begin{align}
		\be_{j+2}p_{j+1}&\prn{\underline{\bT}}\be_1 
		=  
		\be_{j+2}\prn{\underline{\bT}-\lambda_{j+1}\bI}p_j\prn{\underline{\bT}}\be_1
		\nonumber
		\\
		= &
		\begin{bmatrix}
			0 \cdots 0 & 
			\alpha_{j+2}\beta_{j+2} & 
			\alpha_{j+2}^2+\beta_{j+3}^2-\lambda_{j+1} & 
			\alpha_{j+3}\beta_{j+3} & 0 \cdots
		\end{bmatrix}
		p_m\prn{\underline{\bT}}\be_1,
		\label{eq:tridiag-induction-dot-product}
	\end{align}
	wherein the only nonzero entries correspond to entries $j+1$, $j+2$, and $j+3$ of $(j+2)$nd row of $\underline{\bT}$. It is
	well known that $\underline{\bT}$ encodes a nearest-neighbor coupling of an infinite one-dimensional lattice. Taking
	increasing powers of $\underline{\bT}$ increases the coupling distance by one per power increase; e.g.,
	$\underline{\bT}^2$ has two super- and two sub-diagonal bands; $\underline{\bT}^3$ has three super- and three sub-diagonal
	bands; and generally, $\underline{\bT}^m$ has $m$ super- and $m$ sub-diagonal bands.  Since $p_m\prn{t}$ is monic and of
	degree $m$, we know that the outermost $m$-th sub- and super-diagonal have entries only coming from the leading monic term
	$\underline{\bT}^m$.  To compute \eqref{eq:tridiag-induction-dot-product}, we are only interest in how the row forms a dot
	product with the first column of $\underline{\bT}^m$.  Because of the aforementioned banding structure, we observe that
	the first column is only nonzero until the 
	\linebreak
	$(m+1)$-st row.  Thus, we conclude $\be_{j+2}^T
	p_{j+1}\prn{\underline{\bT}}\be_1 = \alpha_{m+2}\beta_{m+2}\prod_{i=2}^{m+1}\alpha_i\beta_i$. Combining
	\eqref{eq:poly-oper-tridiag}, \eqref{eq:poly-op-lower-bound}, and \eqref{eq:tridiag-induction-dot-product} completes the
	proof.
\end{proof}
If follows directly from \Cref{lemma:T-offdiagonal-bound} that the off-diagonal entries of $\underline{\bT}$ exhibit a decay pattern.
\begin{corollary}
	It holds that $\lim_{j\rightarrow\infty}\prod_{i=2}^jb_i = 0$.
\end{corollary}
\begin{proof}
	From \Cref{lemma:T-offdiagonal-bound}, the product of the first $j-1$
	subdiagonal entries is bounded by a constant times the product of the
	first $j$ eigenvalues of a self-adjoint compact operator on a Hilbert
	space.  This product of eigenvalues necessarily decays to zero in the
	limit.  Thus, so does $\prod_{i=2}^j\alpha_i\beta_i$.
\end{proof}

\begin{remark}[Relation of Lanczos and \gkb to the singular system]
	\label{remark:gkb-singular-system-relationship}
	We would be remiss if we did not discuss the connections in numerical
	linear algebra between the \gkb and Lanczos iterations with the \svd
	and the approximation of singular values and vectors; see, e.g.,
	\cite[Section 10.4]{GolubVanLoan:2013:1} and \cite[Chapter
	6]{Saad:2011:1} and references cited therein. A detailed description
	and analysis of these techniques is beyond our scope. However, the
	basic idea is that at iteration $m$, one computes the \svd $\bB_m =
	\bF_m\bSigma_m\bG_m^T$ and takes $\prn{U_m\bF_m, \bSigma_m, V_m\bG_m}$
	as an approximation of a part of the singular system, with accuracy
	being governed by $m$ and properties of the operator.  
\end{remark}
\section{Filtration of the Lanczos vectors}
\label{section:filtration-lanczos-vectors}
The theory we have built up in this manuscript demonstrates that the \cgt iterates can be expressed as \dquotes{filtrations} of
the Lanczos directions of the corresponding \cgne iterates.  This is of interest both for better understanding the relationship of
\cgt and \cgne and for presenting a framework of Krylov basis vector filtration that may be useful for understanding the behavior
of existing methods and for developing new methods.

\subsection{Expressing \cgt as a filtration of the \cgne Lanczos directions}
We observe that the expressions for the $m$-th \cgne \eqref{eq:CGNE-Lanczos-basis} and \cgt iterates 
\eqref{eq:CGT-Lanczos-basis} in terms of the Lanczos coefficients share common factors in the coefficients.  Both sets of
coefficients have terms of the form $\prod_{i=2}^jb_i$, which decay in a manner controlled by the eigenvalues as shown in
\Cref{lemma:T-offdiagonal-bound}.  Thus, this product plays no role in relating the \cgne and \cgt coefficients.  The rest
can be related via algebraic manipulations.
\begin{lemma}
	Let 
	\begin{align*}
		\omega_i^{(m)} 
		= 
		\be_i^T\by_m
		\qquad
		\mbox{and}
		\qquad
		\omega_i^{(c,m)} 
		= 
		\be_i^T\by_m^{(c)}
	\end{align*}
	be the $i$-th \cgne and \cgt coefficients of the Lanczos vectors when the $m$-th iterates are expanded in the Lanczos
	basis.  Then we have 
	\begin{align*}
		\omega_i^{(c,m)}
		=
		\prn
		{
			1
			-
			\dfrac
			{g_m(c)\phi_{i+1}^{(m)} - h_{i+1}^{(m)}\det\bT_m}
			{\phi_{i+1}^{(m)}\prn{\det \bT_m + g_m(c)}}
		}
		\omega_i^{(m)}
	\end{align*}
\end{lemma}
\begin{proof}
	We manipulate \eqref{eq:CGT-Lanczos-basis} to prove the result, whereby it follows that
	\begin{align*}
		\omega_i^{(c,m)}
		=&
		\dfrac{1}{\det\prn{\bT_m + c\bI}}
		(-1)^{i+1}
		b_1\cdots b_i
		\prn{\phi_{i+1}^{(m)} + h_{i+1}^{(m)}(c)}
		\\
		=&
		\dfrac{1}{\det\bT_m + g_m(c)}
		(-1)^{i+1}
		b_1\cdots b_i
		\prn{\phi_{i+1}^{(m)} + h_{i+1}^{(m)}(c)}	
		\\
		=&
		\dfrac{1}{\det\bT_m + g_m(c)}
		\prn
		{
			(-1)^{i+1}
			b_1\cdots b_i
			\dfrac{\phi_{i+1}{(m)}}{\det\bT_m}
		}
		\prn{\phi_{i+1}^{(m)} + h_{i+1}^{(m)}(c)}	
		\dfrac{\det\bT_m}{\phi_{i+1}{(m)}}	
		\\
		=&
		\dfrac
		{
			\det\bT_m
			\prn
			{
				\phi_{i+1}^{(m)} + h_{i+1}^{(m)}(c)
			}
		}
		{
			\phi_{i+1}^{(m)}
			\prn
			{
				\det\bT_m + g_m(c)
			}
		}
		\omega_i^{(m)}.
	\end{align*}
	Some further manipulation to pull a $1$ out of the fractional coefficient yields the result.
\end{proof}
As a shorthand, we denote 
\begin{align}
	\gamma_i^{(m)}(c) 
	=&
	\prn
	{
		1
		-
		\frac
		{g_m(c)\phi_{i+1}^{(m)} - h_{i+1}^{(m)}(c)\det\bT_m}
		{\phi_{i+1}^{(m)}\prn{\det \bT_m + g_m(c)}}
	}
	\nonumber
	\\
	=&
	1 
	-
	\frac{
		g_m(c)
	}
	{
		\det \bT_m + g_m(c)
	}
	- 
	\frac{
		h_{i+1}^{(m)}(c)\det \bT_m
	}
	{
		\phi_{i+1}^{(m)} 
		\prn{
			\det\bT_m + g_m(c)
		}
	}
	\label{eqn:gamma-Lanczos-filter-definition}
\end{align}
so that we can write $\omega_i^{(c,m)}=\gamma_i^{(m)}(c)\omega_i^{(m)}$. From
the second line of \eqref{eqn:gamma-Lanczos-filter-definition} combined with
\Cref{corollary.f-g-poly}, we conclude that the first term approaches $1$ as
$c\rightarrow \infty$ and to $0$ as $c\rightarrow 0$, and the second term
approaches $0$ as $c\rightarrow \infty$ and as $c\rightarrow 0$.  The transient
behavior between these extremes is determined by the Lanczos coefficients, and
thus by the operator and right-hand side.

We note that we can infer from \Cref{remark:gkb-singular-system-relationship}
that any right singular vector having highly non-trivial inner product with
Lanczos direction $v_i$ will be partially damped by $\gamma_i^{(m)}(c)$.
Detailed convergence analysis has been well-documented; see, \eg, \cite[Section
6.6]{Saad:2011:1} and references cited therein.
\subsection{A tale of two residuals}
It is important to understand, as well, how the residuals of \cg and \cgt
relate, since stopping criteria such as the discrepancy principle rely on the
norm of the residual.  However, we observe that there are actually two types of
residual to consider.  As we are solving the normal equations \eqref{eqn.ne},
the the \cg and \cgt iterates have a residual associated with these equations,
which lives in the solution space $\cX$.  Additionally, there is also the
residual connected to the original approximation problem \eqref{eqn.AxAppy},
which lives in the data space $\cY$ and that we call the \dquotes{natural}
residual.

In \cite[Lemma 2.2]{FrommerKahlLippertRittich:2013:1}, it is observed,
following from \cite{PaigeParlettVanDerVorst:1995:1} that the normal equations
residual satisfies
\begin{align}\label{eqn.CG-resid}
    A^{\ast}\prn{y - Ax_m }
	=
	-b_1 b_m \prn{\bT_m ^{-1}}_{m1}v_{m+1}.
\end{align}
It then follows that the \cg-based Tikhonov residual must also 
satisfy a similar relation, namely
\begin{align}\label{eqn.CG-Tikh-resid}
    \prn{A^{\ast}y - (A^{\ast}A + c I_{\cX})x_m ^{c}}
	=
	-b_1 b_m \prn{ \prn{\bT_m  + c\bI}^{-1}}_{m1}v_{m+1}.
\end{align}
From \eqref{eqn.Tm-inverse}, we see that 
\begin{align}
    \prn{\bT_m ^{-1}}_{m1}
	=
	\dfrac{\prn{-1}^{m+1} b_2 \cdots b_m \theta_0 \phi_{m+1}^{(m)}}{\theta_m }=\dfrac{\prn{-1}^{m+1} b_2 \cdots b_m }{\theta_m }.
\end{align}
Similarly, we have for the shifted system, 
\begin{align}
    \prn{ \prn{\bT_m  + c\bI}^{-1}}_{m1}
	=
	\dfrac{\prn{-1}^{m+1} b_2 \cdots b_m }{\theta_m  + g_m(c)}.
\end{align}
From this we can conclude,
\begin{lemma}
    The residuals \eqref{eqn.CG-resid} and \eqref{eqn.CG-Tikh-resid}
    can be related by 
    \begin{align}\label{eqn.resid-relation}
        \prn{A^{\ast}y - (A^{\ast}A + c I_{\cX})x_m ^{(c)}}= \dfrac{\theta_m }{\theta_m  + g_m(c)} A^{\ast}\prn{y - Ax_m }
    \end{align}
\end{lemma}
Consider that these residual expressions are ones mapped to $\cX$
by the adjoint operator $A^{\ast}$, and 
$\prn{A^{\ast}y - (A^{\ast}A + c I_{\cX})x_m ^{c}}$ is only
interesting as a measure of convergence of the \cg iteration.  The residuals
we are interested for analysis purposes are the natural residuals
\begin{align}
    r_m 
	=
	y - Ax_m \qquad\mbox{and}\qquad r_m ^{(c)}
	=
	y - Ax_m ^{(c)}.
\end{align}
With these definitions, we can rewrite \eqref{eqn.resid-relation} as
\begin{align}
    A^{\ast}r_m 
	=
	\dfrac{\theta_m  + g_m(c)}{\theta_m }\prn{A^{\ast}r_m ^{(c)} - c x_m ^{(c)}}.
\end{align}
This is interesting but still only tells us about the relationship of
the residuals when they are mapped to the solution space $\cX$ by $A^\ast$. 
We can instead use the Golub-Kahan relation to obtain a representation of the residual
in terms of the Lanczos filters.
\begin{corollary}
	The $m$-th \cgt residual $r_m^{(c)} = y^\delta - Ax_m^{(c)}$ admits the representation
	\begin{align}
		r_m^{(c)} 
		=
		\prn{
			\beta_0 
			- 
			\gamma_1^{(m)}(c)\omega_1^{(m)}\alpha_1 
		}
		u_1
		-&
		\sum_{i=2}^m
		\prn{
			\gamma_i^{(m)}(c)\omega_i^{(m)}\alpha_i 
			- 
			\gamma_{i-1}^{(m)}(c)\omega_{i-1}^{(m)}\beta_i
		}
		u_i
		\nonumber
		\\
		-&
		\gamma_m^{(m)}(c)\omega_m^{(m)}\beta_{m+1}
		u_{m+1},
		\label{eqn:resid-rep}
	\end{align}
	and it follows that 
	\begin{align*}
		\norm[auto]{r_m^{(c)}}^2
		=
		\prn{
			\beta_0 
			- 
			\gamma_1^{(m)}(c)\omega_1^{(m)}\alpha_1 
		}^2
		+&
		\sum_{i=2}^m
		\prn{
			\gamma_i^{(m)}(c)\omega_i^{(m)}\alpha_i 
			- 
			\gamma_{i-1}^{(m)}(c)\omega_{i-1}^{(m)}\beta_i
		}^2
		\nonumber
		\\
		+&
		\prn{
			\gamma_m^{(m)}(c)\omega_m^{(m)}\beta_{m+1}
		}^2
	\end{align*}
\end{corollary}
\begin{proof}
	This result is obtained by inserting the expression $x_m^{(c)}=\sum_{i=1}^{m}\gamma_i^{(m)}(c)\omega_i^{(m)}v_i$ into the
	residual and applying the Golub-Kahan relation \eqref{eqn:GK-3-term-recur}.  Then one simply collects terms for each
	$u_i$.  That the vectors $\paren[auto]{\{}{\}}{u_i}$ form an orthonormal basis yields the square norm immediately.
\end{proof}
\subsection{More general notions of filtration}
One observes that this analysis allows us to develop expressions for the specific Lanczos filters $\gamma_i^{(m)}(c)$ associated
to \cgt.  However, this notion of Lanczos filtration can be divorced from \cgt-specific derivations.  One can instead consider
strategies directly stemming from the filtration process.  For example, for the \texttt{shaw\_chebfun()} problem introduced in
\Cref{subsection:organization}, the discrepancy principle stopping criterion halts the iteration at $m_{discr}=7$.  Thus, for any
$m>m_{discr}$, it follow that we can express $x_{m_{discr}} = \sum_{i=1}^m \gamma_i^{(discr)}\omega_i^{(m)}v_i$ where
$\gamma_i^{(discr)}=0$ for $i>7$. In essence, the filters $x_{m_{discr}}$ with respect to $x_m$ are almost truncation filters,
similar to those for truncated \svd with respect to the singular system.  However, as demonstrated in
\Cref{figure:shaw-lanczos-filters}, this is not quite the case since $\gamma_i^{(discr)}$ for $i\leq 7$ are not necessarily $1$.
The general use of Lanczos filters (or filters applied to other bases for different types of iterative methods) to analyze or
develop new methods is beyond the scope of the present work.  In \Cref{section.conclusions}, we discuss this a bit further in the
context of future work.
\begin{figure}
	\begin{centering}
		\includegraphics[scale=0.3]{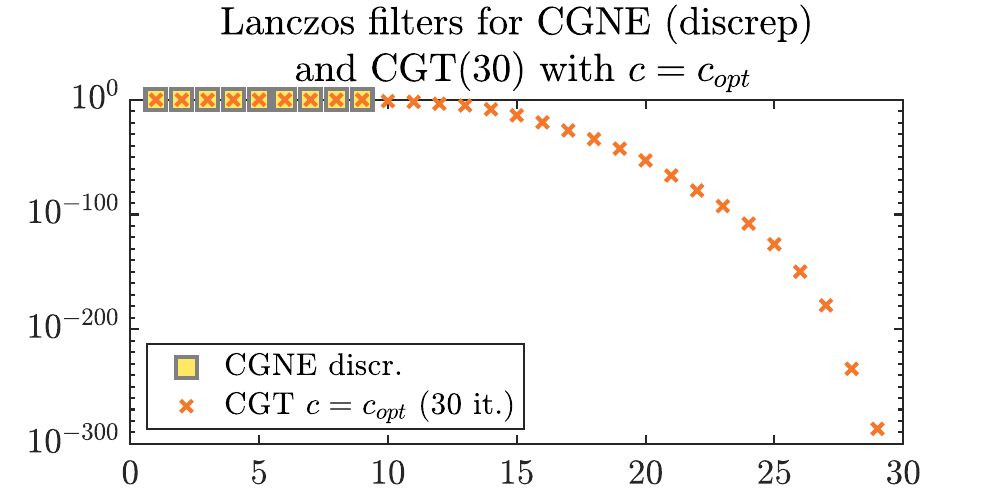}
		\includegraphics[scale=0.3]{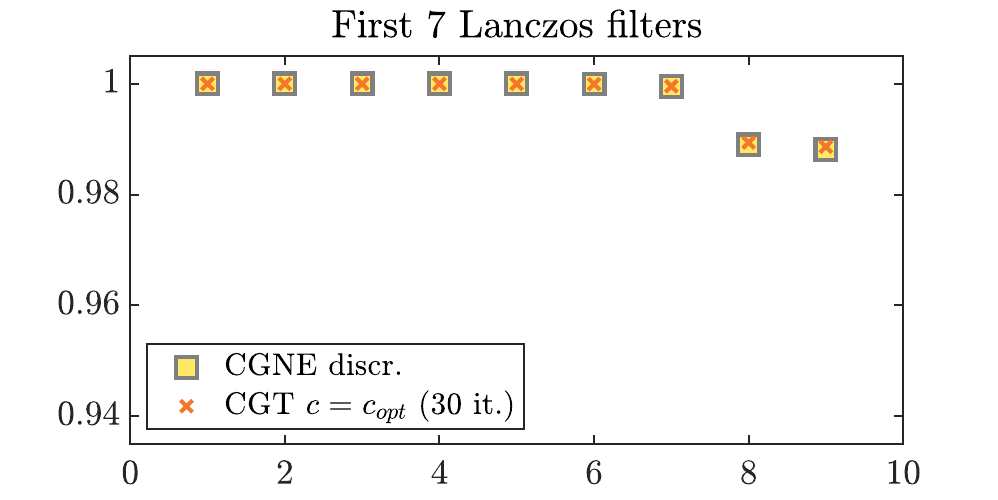}
	\end{centering}
	\caption{\label{figure:shaw-lanczos-filters}Lanczos filters for discrepancy principle iterate
	$m_{discr}=7$ and \cgt $c=c_{opt}$ with respect to \cgne(30) iterate.}
\end{figure}
\subsection{Dampening of propagated noise}
It is well documented that running too many iterations of \cgne applied to an ill-posed problem can amplify aspects of the noise to
pollute the reconstructed solutions.  One can understand this, e.g., by observing that it has been proven that \cgne converges to
the pseudoinverse solution $x^\dagger$; see e.g., \cite[Theorem 7.9]{EnglHankeNeubauer:1996:1} and references therein.  The
singular system representation \eqref{eqn:SVD-pseudoinverse-soln} shows that arbitrarily small perturbations of $y$ lead to
unbounded perturbations of $x^\dagger$. This necessarily means that the \gkb vectors will increasingly contain the amplified noise
present in later singular directions.  Thus, the Lanczos filters can be understood to be dampening the influence of these vectors.
Thus, for a given $y^\delta$, we expect the $x_m^\delta = \sum_{i=1}^m \omega_i^{(m)} v_i$ to be unbounded as $m\rightarrow\infty$. In
\Cref{fig:shaw-varTikh-lanczos-filters}, we observe that even for small values of the Tikhonov parameter, the Lanczos filters
decay rapidly for the later bidiagonalization vectors. 
\begin{figure}
	\begin{center}
		\includegraphics[scale=0.3]{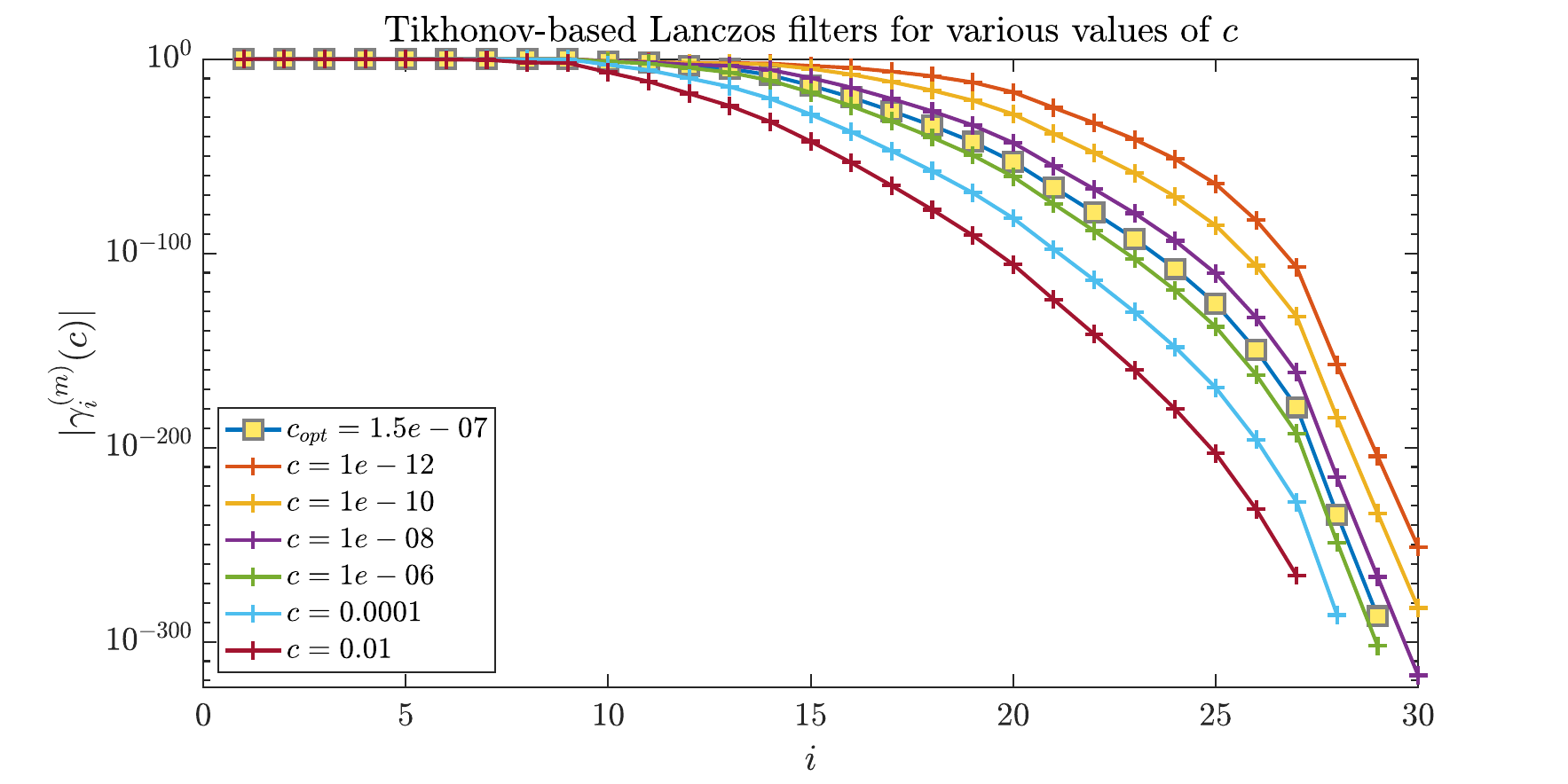}
	\end{center}
	\caption{\label{fig:shaw-varTikh-lanczos-filters}Demonstration of the Lanczos filter dampening effect.}
\end{figure}
\begin{remark}
	We note that the result from \cite[Theorem 7.9]{EnglHankeNeubauer:1996:1} indicates that \cgt for any value of $c$ will
	also converge to the pseudoinverse solution of the Tikhonov problem; i.e., it converges to the Tikhonov solution
	$x^{(c)}$ for a given $c$.  As the Tikhonov problem is by construction well-posed, it follows that $x^{(c)}$ is bounded
	for any $c>0$. We can thus represent the Tikhonov solution
	\begin{align*}
			x^{(c)} = \sum_{i=1}^\infty \omega_i^{(c)} v_i := \lim_{m\rightarrow\infty}\sum_{i=1}^m\omega_i^{(m,c)} v_i,
	\end{align*}
	and we note that the convergence of this series indicates that coefficients $\omega_i^{(c)}$ must exhibit sufficiently
	rapid (in the $\ell_2$ sense) decay behavior.  Thus, the results we have shown concerning the Tikhonov-based Lanczos
	filtration of the \cg iterates is not, in itself, a surprising result.  Rather, it is the interpretation of \cgt as a
	filtration of the Lanczos vectors and the precise nature of these filters that is of interest.
\end{remark}

Restricting to the finite-dimensional, discrete ill-posed setting allows for clearer illustration of this behavior. Such problems
arise from the discretization of an ill-posed problem \eqref{eqn.AxAppy}, e.g., $\bA_h\bx_h \approx \by_h$ wherein the resulting
matrix $\bA_h\in\R^{m\times n}$ is \emph{discretely ill-posed}. This means its singular values decay rapidly with no clear break
between \dquotes{good} and \dquotes{bad} singular values.  In this case, the pseudoinverse solution is not unbounded, but it
departs significantly from the true solution such that it is often meaningless. In \cite{HnetynkovaPlesingerStrakos:2009:1}, it is
shown under some mild assumptions that the Golub-Kahan bidiagonalization is a discrete noise-revealing iterative method, as the
noise propagates and is amplified as the iteration progresses. This can be used to estimate the noise level when it is not known
apriori.  This idea is extended further in \cite{HnetynkovaKubinovaPlesinger:2017:1} to understand how the discrete noise
propagates in \gkb-based iterations such as \lsqr.  Indeed, the authors show that as the iteration progresses, the noise is
propagated and amplified via the left bidiagonalization vectors into the residual.  Thus, we can interpret the Tikhonov-based
Lanczos filtration as dampening the Lanczos directions that propagate the amplified noise. 
\section{A noisier numerical example}
\label{section.numerical-examples}
In the derivation of the ideas in this work, we already present one set of numerical experiments for the \texttt{shaw} problem.
This work is proof-of-concept in nature, meant to serve as a foundation for future development of practical methods based on the
ideas presented, and the experiments demonstrate the ideas. Indeed we use the theoretically optimal (and practically unavailable) 
Tikhonov parameter in our experiments, to this end. 

With this in mind, we provide a further demonstration in this section, using the \texttt{gravity} example from the
chebfun-based toolbox introduced in \cite{AlqahtaniMachReichel:2022:1}.  The \texttt{gravity} problem as a one-dimensional
Fredholm (first kind) integral equation model from \dquotes{gravity surveying} with kernel 
$k(s,t) = \dfrac{1}{4}\cdot\prn{\dfrac{1}{4}^2 + (s-t)^2}^{-3/2}$.
We increase the noise level in this experiment to $10^{-2}$. 
In the toolbox, this example has options to provide different right-hand sides; we choose the second option, a piecewise linear
solution.  Noise is generated as before using the \chebfun random function generation procedure.  In \Cref{fig:gravity-setup}
\begin{figure}
	\centering
	\includegraphics[scale=0.4]{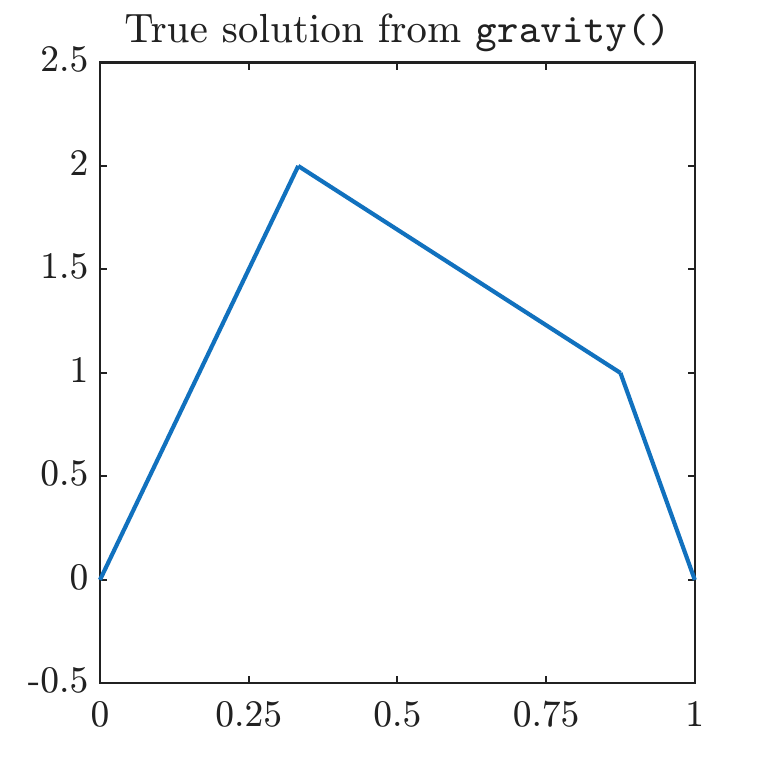}
	\includegraphics[scale=0.4]{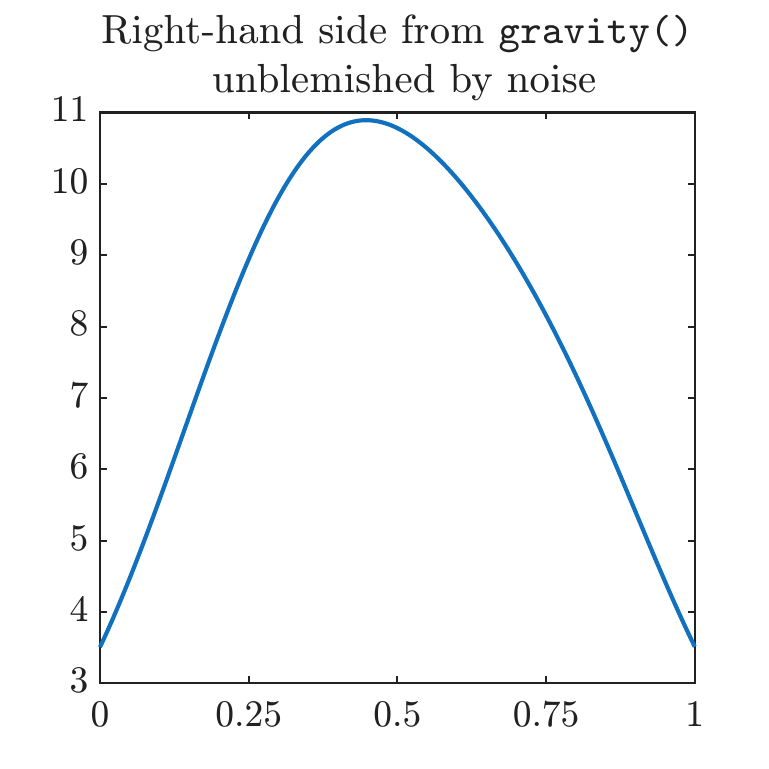}
	\\
	\includegraphics[scale=0.4]{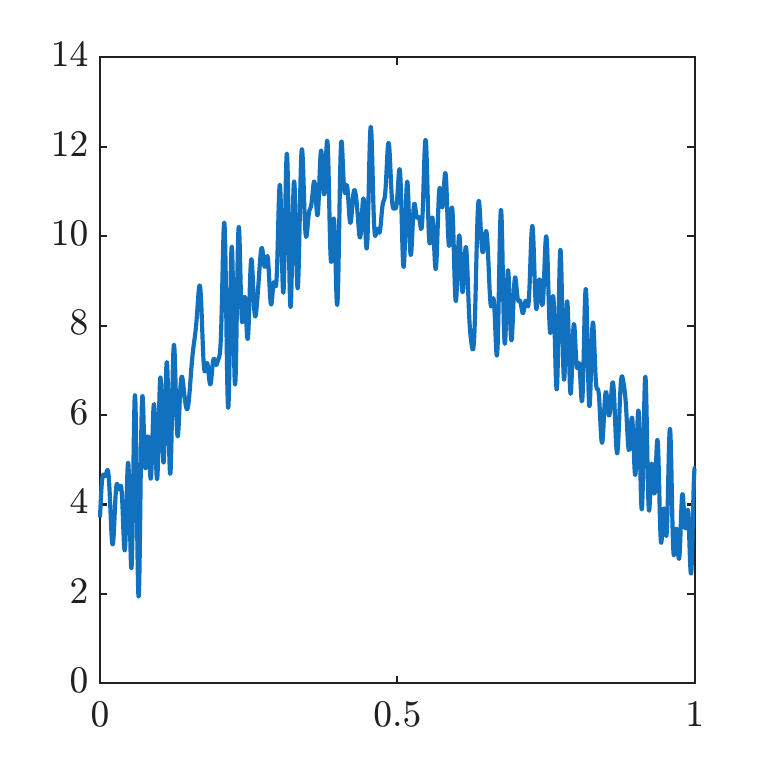}
	\includegraphics[scale=0.4]{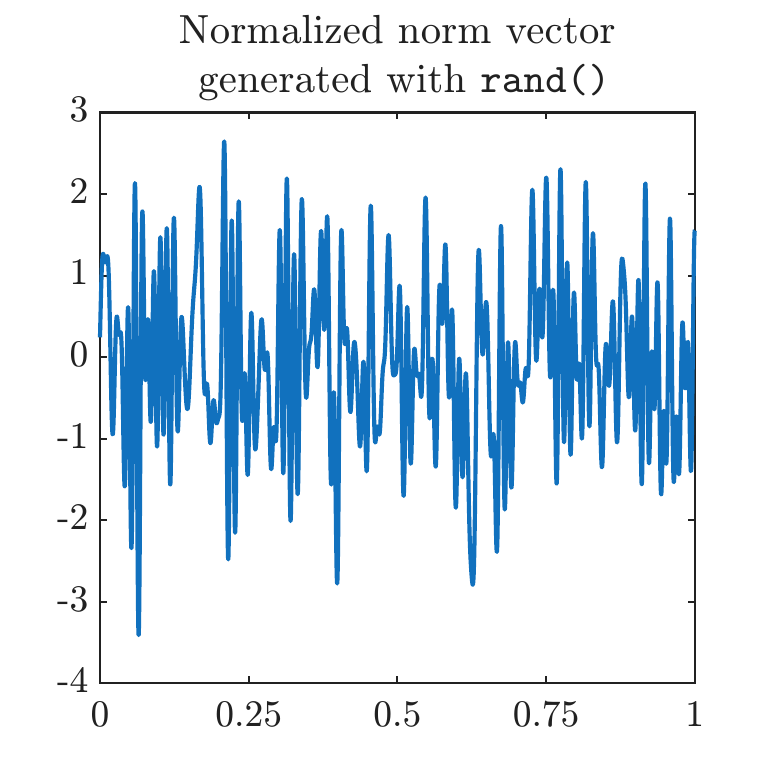}
	\caption{\label{fig:gravity-setup}The true piecewise linear solution and the right-hand side (unblemished and perturbed
	with noise) and the noise for the \texttt{gravity} problem with relative noise level $10^{-2}$.}
\end{figure}
we show the exact, piecewise linear solution and the unperturbed right-hand side, along with the noisy right-hand side and the normalized
noise.

We choose this example because Krylov-type iterative approaches often have difficulties resolving areas of non-smoothness of the
true solution, since they are built from a noisy version of the right-hand side (which is often smooth) and powers of $A^\ast A$
applied to it.  Indeed, we see in \Cref{fig:gravity-solves} that both the \cg discrepancy solution and the optimal \cgt solution that
both attempt to smoothly reconstruct the solution and are unable to resolve the piecewise transitions between linear
parts of the solution.  Due to the level of noise, the discrepancy principle stops \cgne after two iterations, but optimal \cgt
damps the coefficients of \cgne in order to be able to further refine the solution approximation in a meaningful way. 
\begin{figure}
	\centering
	\includegraphics[scale=0.35]{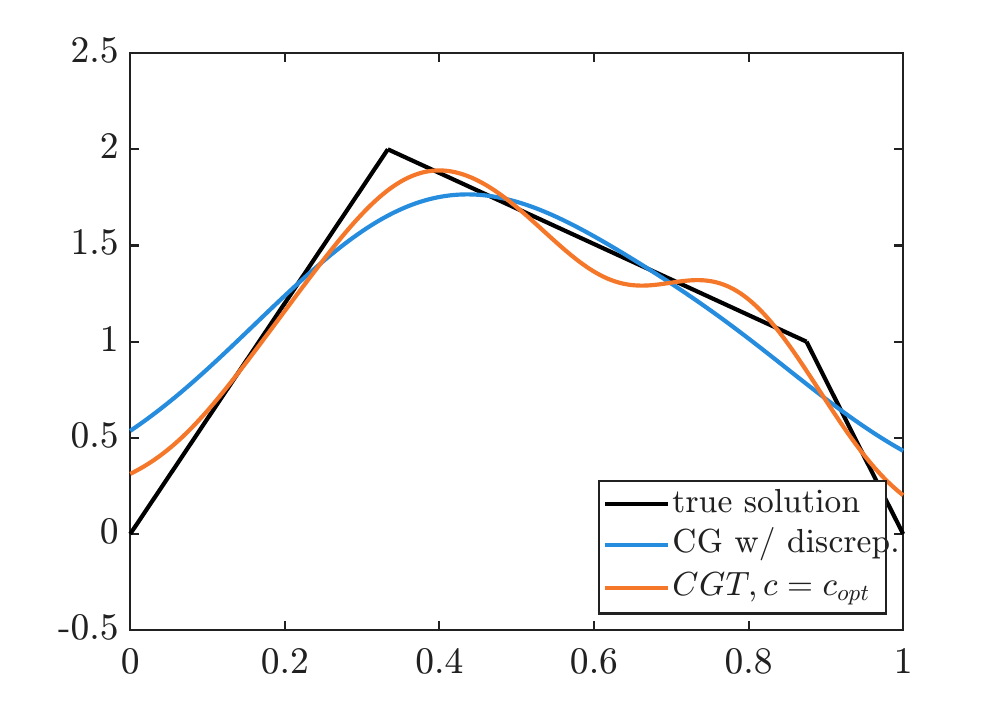}
	\includegraphics[scale=0.35]{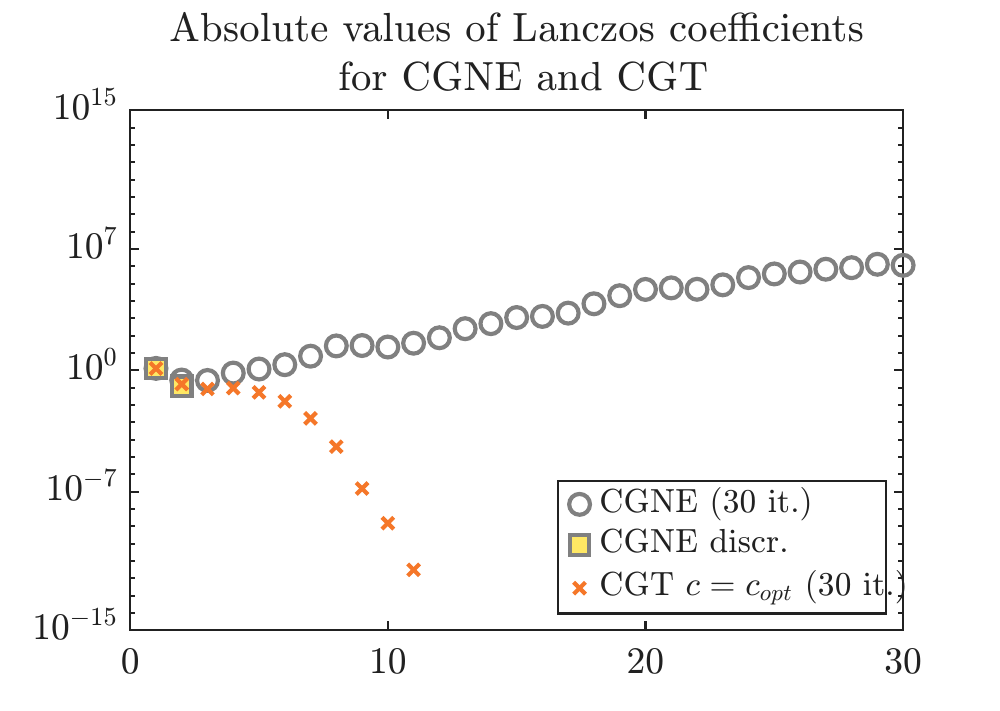}
	\caption{\label{fig:gravity-solves} We compare the true solution with the \cg discrepancy solution and optimal
	\cgt solution. Note: we display absolute values of coefficients to show them all in a single \texttt{semilogy} 
	plot.}
\end{figure}
\begin{remark}
	We conclude by noting that although \chebfun allows us to simulate running these methods for infinite dimensional,
	compact problems, this is still a high-precision, finite-dimensional approximation to the true infinite dimensional
	problem.  Thus, the \cgne iteration will still converge to a bounded, albeit highly inaccurate, meaningless approximation
	of $x$.
\end{remark}
\section{Conclusions and future work}
\label{section.conclusions}
In this work, we have analyzed the relationship between the \cgne and \cgt iterates, and we have used two one-dimensional
numerical examples to illustrate this behavior and our interpretation of it.  Our interpretation of this analysis is that, in the
Lanczos basis (\ie the iteration basis), \cgt can be expressed as the \cgne iterate with damping factors applied to the
coefficients.  These damping factors can be expressed using quantities that are polynomials of the Tikhonov parameter.  More
generally, we observe that this interpretation gives us an alternative to the more classical filtration of singular directions
approach to understand the mechanisms governing the behavior of the \cgt.  Rather than describing the behavior in terms of the
filtration of the representation of the pseudoinverse solution in terms of the singular vectors of the operator governing the
ill-posed problem, we can instead understand how it filters the \cgne approximation of the pseudoinverse solution expressed in
terms of the Lanczos basis.

This work began as an exploration of the relationship between \cg and \cgt, but our inquiries led us to develop the notion of
Lanczos filtration.  The work thus far is simply proof-of-concept.  There is much more that can be built on this foundation. There
are much more complicated and advanced hybrid iterative approaches for treating ill-posed problems \cite{ChungGazzola:2024:1}, and
iteration basis filtration offers another approach for analyzing these methods and understanding how they work.  Furthermore, we
observe that one can consider methods of iteration basis filtration not defined through the minimization of a Tikhonov-type
functional. We may also consider extending these ideas to hybrid methods for sums of Krylov and non-Krylov subspaces; see \eg
\cite{JiangChungDeSturler:2021:1,RamlauSoodhalterHutterer:2021:1,SoodhalterDeSturlerKilmer:2020:1}.  Lastly, our analysis may offer a different approach to developing parameter-choice rules.
\section*{Acknowledgement}
The second author would like to thank Gerard Meurant for a productive
discussion on this topic and for sharing a chapter of an early draft of his
book \cite{Meurant:2025:1} on Hessenberg matrices.  The second author would
also like to thank Lothar Reichel for recommending the the authors use the
\texttt{chebfun} code from \cite{AlqahtaniMachReichel:2022:1} to build their
experiments.

\appendix

\printbibliography

\end{document}